\documentclass{article}

\usepackage[utf8]{inputenc}
\usepackage{amsfonts}
\usepackage{comment}
\usepackage{xcolor}

\usepackage{amsthm,amsmath,amssymb}
\usepackage{mathrsfs}
\usepackage[numbers]{natbib}  
\usepackage[colorlinks,citecolor=blue,urlcolor=blue]{hyperref}  

\usepackage{enumerate, enumitem}

\usepackage{tgtermes}
\usepackage{mathptmx} 
\usepackage[scaled=.92]{helvet}

\numberwithin{equation}{section}

\newtheorem{theorem}{Theorem}[section]
\newtheorem{lemma}[theorem]{Lemma}
\newtheorem{proposition}[theorem]{Proposition}
\newtheorem{corollary}[theorem]{Corollary}
\newtheorem{definition}[theorem]{Definition}
\newtheorem{example}[theorem]{Example}

\newtheorem{remark}[theorem]{Remark}

\newcommand{\N}{\mathbb{N}}

\newcommand{\Q}{\mathbb{Q}}
\newcommand{\R}{\mathbb{R}}

\newcommand{\F}{\mathbb{F}}

\newcommand{\hR}{\,\!^\ast\R}
\newcommand{\ns}[1]{\,\!^\ast#1} 
\newcommand{\sh}[1]{\,\!^\circ#1} 

\newcommand{\civita}{\mathcal{R}}

\newcommand{\an}{C^{\omega}}

\newcommand{\ext}[2]{\overline{#2}_{#1}}

\newcommand{\st}{\mathrm{st}}

\renewcommand{\P}{\mathcal{P}}

\newcommand{\norm}[1]{\left\Vert#1\right\Vert}

\DeclareMathOperator*{\wlim}{w-lim}

\DeclareMathOperator*{\supp}{supp}
\newcommand{\m}{m}
\newcommand{\meas}{\mathcal{M}}

\newcommand{\leb}{\lambda}

\renewcommand{\L}{\mathcal{L}}

\newcommand{\loeb}{\m_L}
\newcommand{\extl}{\overline{\m}_L}

\newcommand{\epi}{\mathcal{U}}
\newcommand{\civint}{{$M$-integrable}}
\newcommand{\civintegral}{{$M$-integral}}

\begin{document}
	\title{A real-valued measure on non-Archimedean field extensions of $\R$}
	\author{Emanuele Bottazzi}
	\date{\today}
	\maketitle
	
\begin{abstract}
	We introduce a real-valued measure $\loeb$ on non-Archimedean ordered fields $(\F,<)$ that extend the field of real numbers $(\R,<)$.
	The definition of $\loeb$ is inspired by the Loeb measures of hyperreal fields in the framework of Robinson's analysis with infinitesimals.
	The real-valued measure $\loeb$ turns out to be general enough to obtain a canonical measurable representative in $\F$ for every Lebesgue measurable subset of $\R$, moreover the measure of the two sets is equal. In addition, $m_L$ it is more expressive than a class of non-Archimedean uniform measures.
	We focus on the properties of the real-valued measure in the case where $\F=\civita$, the Levi-Civita field.
	In particular, we compare $\loeb$ with the uniform non-Archimedean measure over $\civita$ developed by Shamseddine and Berz, and we prove that the first is infinitesimally close to the second, whenever the latter is defined.
	We also define a real-valued integral for functions on the Levi-Civita field, and we prove that every real continuous function has an integrable representative in $\civita$.
	Recall that this result is false for the current non-Archimedean integration over $\civita$.
	The paper concludes with a discussion on the representation of the Dirac distribution by pointwise functions on non-Archimedean domains.
\end{abstract}

\tableofcontents

\section{Introduction}
Measure theory on non-Archimedean fields, in particular non-Archimedean extensions of $\R$ (and not e.g. fields of $p$-adic numbers) holds the promise to be relevant for many applications. An example is the mathematical description of physical phenomena, for instance through the representation of distributions and Young measures as pointwise functions over non-Archimedean fields (for more details, we refer to the discussion in \cite{grid functions}; some recent examples are \cite{lpcivita,shamflin2}). Another relevant application is differential or algebraic geometry, as discussed for instance in \cite{berarducci-otero,kaiser}. However, currently the measure theory of non-Archimedean fields is limited to some particular extensions of $\R$ or to a more restricted class of sets than e.g.\ the $\sigma$-algebra generated by intervals.

A particular class of non-Archimedean extensions of $\R$ where a sizeable measure theory has already been developed is that of hyperreal fields of Abraham Robinson's framework of analysis with infinitesimals \cite{robi,robi1}.
In this context the elementary equivalence of $\hR$ and $\R$ and the presence of the \emph{transfer principle} allow for a rich measure theory.
An immediate consequence of the transfer principle applied to any real-valued measure $\mu$ is that $\ns{\mu}$ is a hyperreal-valued set function satisfying the following conditions
\begin{itemize}
	\item $\ns{\mu}$ is non-negative;
	\item $\ns{\mu}$ is monotone;
	\item $\ns{\mu}$ is finitely additive (and, since $\ns{\mu}$ is internal, it is also \emph{hyperfinitely additive}).
\end{itemize}
Notice that in general the hyperreal measure $\ns{\mu}$ is not $\sigma$-additive, even if $\mu$ is \cite{cutloeb}. 
For this reason, most of the initial works on hyperreal measure theory were focused on hyperfinitely additive measures.
For some applications, this limitation turns out not to be restrictive, since hyperfinitely additive measures are general enough to represent every real non-atomic measure, including those that are $\sigma$-additive \cite{bbd,henson}.
Moreover, it is possible to represent uncountably many real-valued measures with a single hyperfinitely additive measure \cite{watt}, or to require further compatibility conditions between the non-Archimedean measure and the real-valued measure it represents \cite{bbd, bbd2}.
In fact, these results are true even if one works with the family of \emph{hyperfinite counting measures}, i.e. measures of the form $\mu(A) = \frac{|A|}{|\Omega|}$ where $\Omega$ is a hyperfinite set, $A \in \ns{\P(\Omega)}$ (i.e. $A$ is an internal subset of $\Omega$) and $| \cdot |$ denotes the internal cardinality.

A novel contribution to hyperfinite measures has been introduced by Eskew in a recent preprint \cite{eskew}. Eskew defines an \emph{ultrafilter integral} of functions $f: X \rightarrow G$, where $X$ is an {arbitrary} set and $G$ is a divisible Abelian group.
The ultrafilter integral depends upon the choice of an ultrafilter $U$ over the family $\P_{fin}(X)$ of finite subsets of $X$ and takes value in a nonstandard extension of the group $G$ (more precisely, in the ultrapower $\prod_{i \in \P_{fin}(X)} G / U$).
By defining a \emph{standard part} from this ultrapower to the original group $G$, it is possible to define a $G$-valued integral for \emph{every} function $f: X \rightarrow G$.
If $G=\R$, this integral is general enough to represent every non-atomic measure, similarly to the case of hyperfinite measures \cite{henson} and numerosities \cite{bbd,bbd2}.
By considering the integral of indicator functions, this technique can be used to define real-valued measures over arbitrary sets.
We believe that this technique can be suitably rephrased as a hyperfinite sum taking values in $\ns{G}$, thus providing an extension of the usual hyperfinite measures.
For its generality, ultrafilter integration might be successfully applied to further advance the measure theory on non-Archimedean fields.

Despite the expressive power of hyperfinite measures, they lack some familiar properties of measures, such as $\sigma$-additivity. The problem of determining a suitable $\sigma$-additive, real-valued measure from an internal measure has been solved by Loeb with the introduction of the \emph{Loeb measures construction}. 
The main idea behind the Loeb measure construction applied to an internal measure $\mu$ consists of the following steps:
\begin{itemize}
	\item define the real-valued set function ${\mu_\R}$ by posing ${\mu_\R}(A) = \sh{\mu(A)}$;
	\item prove that ${\mu_\R}$ is an outer measure on the algebra of internal subsets of $\Omega$;
	\item use the \emph{Caratheodory's extension theorem} to extend ${\mu_\R}$ to a $\sigma$-additive measure $\mu_L$ on a $\sigma$-algebra that extends the algebra of internal subsets of $\Omega$.
\end{itemize}
For more details we refer to the original paper by Loeb \cite{loeb-orig} and to other presentations of the Loeb measure, such as \cite{loeb,cutloeb}.
Since their introduction, Loeb measures have proven to be relevant in a variety of applications.
The earliest examples by Loeb discuss probability theory and stochastic processes \cite{loeb-orig}, but there are further applicatons for instance in the representation of parametrized measures \cite{cutland controls} and in the study of generalized solutions to partial differential equations \cite{grid functions,illposed}.


The development of a measure theory on other non-Archimedean field extensions of $\R$ faces significant challenges. 
One of the most successful projects towards this goal is the uniform measure over the Levi-Civita field defined by Shamseddine and Berz \cite{shamseddine2012,berz+shamseddine2003} and further studied by other authors \cite{lpcivita,shamflin2,moreno,shamflin1}. 
\color{black}
The Levi-Civita field $\civita$, introduced by Levi-Civita in \cite{civita1, civita2} and subsequently rediscovered by many authors in the '900, is the smallest non-Archimedean ordered field extension of the field $\R$ of real numbers that is both real closed and sequentially complete in the order topology.
%
%
The main idea behind the definiton of the uniform measure by Shamseddine and Berz is that measurable sets are those that can be suitably approximated by closed intervals. 
Similarly, measurable functions can be suitably approximated by a family of simple functions. For more details on this topic, we refer to Section \ref{sec integral} and to  \cite{berz+shamseddine2003,lpcivita}.

It turns out that measurable functions on the Levi-Civita field are expressive enough to represent some distributions \cite{lpcivita,shamflin2}.
For instance, it is possible to define some measurable functions that represent the Dirac distribution, much in the spirit of the representation of distributions with functions of nonstandard analysis or other generalized functions (for a detailed discussion on some representations of distributions with these techniques we refer to \cite{grid functions}).
Moreover, it is also possible to represent real continuous functions with suitable equivalence classes of weak limits of measurable functions over $\civita$ \cite{lpcivita}.

However, the uniform measure on the Levi-Civita field has some limitations, mainly due to the total disconnectedness of the topology induced by the non-Archimedean metric.
A first drawback is that the family of measurable set is not closed under complements and over countable unions.
As a consequence, there are well-known examples of null sets whose complement is not measurable \cite{lpcivita,moreno,berz+shamseddine2003}.
In addition, measurable functions are only locally analytic, so in the Levi-Civita field it is not possible to obtain a measurable representative of real continuous functions \cite{lpcivita}.
Finally, in Proposition \ref{proposition shadow of a measurable set} of this paper we will argue that the measurable sets in the Levi-Civita field are not expressive enough to represent all real Lebesgue measurable sets. This is in contrast to the hyperfinitely additive measures that represent the real Lebesgue measure and for the corresponding Loeb measures. 

Another approach to the definition of (real-valued or otherwise) measures over non-Archimedean fields is related to model theory.
One of the earliest results is the definition of a finitely additive real measure on definable sets of o-minimal extensions of fields by Berarducci and Otero \cite{berarducci-otero}. The measurable sets are those definable sets that can be suitably approximated by finite unions of rectangles (this integrability condition is equivalent to the one used in the Caratheodory's extension theorem only under the hypothesis that the measure is finite).
The measure introduced by Berarducci and Otero has also provided a starting point for the development of an Hausdorff measure for definable sets in o-minimal structures \cite{fornasiero hausdorff}.
As observed by Kaiser, these measures are defined for bounded sets and their range is just a semiring \cite{kaiser}.


A significant contribution to the development of a non-Archimedean measure theory on real closed fields $\F$ with a model-theoretic approach is the work by Kaiser \cite{kaiser}.
He introduces a non-Archimedean measure for semialgebraic sets and a corresponding integral for semialgebraic functions over non-Archimedean real closed fields with Archimedean value groups. This measure is finitely additive, monotone and translation invariant. In order to satisfy these properties, the measure takes values outside the field $\F$, since some integrals of semialgebraic functions require a notion of logarithm that is not available for arbitrary real closed fields (for more details on this limitation, we refer to the discussion in \cite{kaiser}).

It is relevant to observe that the measure developed by Kaiser for the Levi-Civita field is not equal to the uniform measure introduced by Shamseddine and Berz, since the former is defined only for semialgebraic sets, while the latter is defined also on some countable unions of intervals, that are not semialgebraic.

The problem of a non-Archimedean measure and integration has also been discussed in the setting of surreal numbers by Fornasiero \cite{fornasiero} and Costin et al. \cite{integral}.
However, most of the results discussed in the latter paper are negative.

Taking into account the existing literature on measures on non-Archimedean fields, some authors suggest that a measure theory on non-hyperreal field extensions of $\R$ requires a tame setting.
Indeed, measures in Robinson's framework are mostly defined on \emph{internal sets}, with the notable exception of the Loeb measures (the \emph{numerosities} by Benci et al. \cite{bbd,bbd2} and the related $\Omega$-limit approach to probabilities \cite{omega,tesi}, on the other hand, use functions defined on the powerset of a classic set with values in a hyperreal field. In both approaches these functions are obtained as the restriction of suitable internal measures, as discussed for instance in \cite{prusso}).
So far, the notion of internal set is only meaningful for hyperreal fields, so that the techniques of Robinson's framework cannot be adapted to other non-Archimedean fields.
In the more general settings of non-Archimedean real closed fields with Archimedean value groups, the measure is defined only for semialgebraic sets, and the integral is defined only for semialgebraic functions.
Finally, in the Levi-Civita field, where there is no notion of {interal set} and the existing non-Archimedean measure has been developed without model-theoretical notions, the family of measurable sets is badly behaved: for instance, we have already mentioned that it is not closed under relative complements.

In this paper, inspired by the success of the real-valued Loeb measure construction and motivated from the consideration that this real-valued measure is not defined only on a well-behaved family of sets (namely, the internal sets), we develop a uniform, real-valued measure for non-Archimedean field extensions of $\hR$.
The main idea is shared with the Lebesgue measure, and consists in defining an outer measure from the length of intervals.
However, we will not consider the length of the interval of endpoints $a$ and $b$ to be equal to $b-a$, but rather to the \emph{standard part} of this difference, namely the real number closest to $b-a$.
This will allow to define an outer measure and, via the  Caratheodory's extension theorem, a corresponding $\sigma$-additive measure.

We will show that this real-valued measure shares some of the properties of the Loeb measures.
For instance, the measure is defined on a $\sigma$-algebra of subsets of $\F$ that is rich enough to represent Lebesgue measurable subsets of $\R$.
It is also possible to extend the real-valued measure to $\F^n$ and, consequently, to define a real-valued integral for functions $f: \F \rightarrow \F$.

The real-valued measure is also compatible with some of the existing measures discussed above.
If $\F=\hR$ is a field of hyperreal numbers, then the real-valued measure agrees with the Loeb measure obtained from the nonstandard extension of the Lebesgue measure over $\R$ (however, it is strictly weaker than the Loeb measure, since e.g. it is not able to assign a positive finite measure to hyperfinite unions of intervals of an infinitesimal length).
If $\F$ is Cauchy complete, then the real-valued measure agrees with the standard part of a non-Archimedean uniform measure that generalizes the one defined by Shamseddine and Berz for the Levi-Civita field.

Finally, we focus on the Levi-Civita field. By adapting the techniques developed in the first part of the paper, we define a real-valued integral on the Levi-Civita field in a way that the corresponding integrable functions are expressive enough to represent real measurable functions.
This result improves upon the previous representation obtained by weakly Cauchy sequences of measurable functions \cite{lpcivita}.
As an application, we improve on previous representations of the Dirac distribution by pointwise functions on non-Archimedean domains.


\color{black}

\subsection{Structure of the paper}

Section \ref{sec loeb} contains the definition of the measure $\loeb$ and of the algebra of $\loeb$-measurable sets. For a matter of convenience, we will refer to $\loeb$-measurable set as \emph{$L$-measurable sets}.
We will show that the the measure $\loeb$ shares some properties with the Lebesgue measure: it is uniform, translation invariant and homogeneous. Moreover, we will show that $\loeb$ can be interpreted as an extension to $\F$ of the real Lebesgue measure. In fact, the main result of this section is the proof that every Lebesgue measurable subset of $\R$ has a canonical $L$-measurable representative in $\F$ with the same measure as the original set.
We also discuss the relation betwen $\loeb$ and $\leb_L$, the Loeb measure obtained from the Lebesgue measure, under the hypothesis that $\F$ is a field of hyperreal numbers. As expected, the Loeb measure is more expressive than the real-valued measure, however the two measures agree on a relevant class of subsets.
Finally, we extend the definition of the real-valued measure to the $n$-dimensional space $\F^n$, and from this definition we introduce a real-valued integral as the measure of the set under the graph of a function. This approach is similar to the introduction of the Lebesgue integral via the $n$-dimensional Lebesgue measure over $\R^n$, presented for instance in \cite{pugh}.

In Section \ref{section comparison with LC} we discuss the relation between the real-valued measure $\loeb$ and a non-Archimedean uniform measure $\m$ on Cauchy complete fields $\F$. This measure is inspired by the one developed for the Levi-Civita field by Shamseddine and Berz. In fact, when $\F=\civita$, then the measure defined in this paper coincides with the one defined by Shamseddine and Berz.
We prove that, if a set $A \subseteq \F$ is $\m$-measurable, then it is also $L$-measurable and $\loeb(A) = \sh{\m(A)}$.
Moreover, we will show that the non-Archimedean measure $\m$ is significantly less expressive than the real-vaued measure $\loeb$, since the projection of $\m$-measurable subsets of $\F$ to $\R$ can be written as a finite union of intervals and of a countable set.

We further pursue the development of a real measure theory on the Levi-Civita field with the introduction of another real-valued integral on the Levi-Civita field in Section \ref{sec integral}.
The definition of this real-valued integral relies on the existing integration theory \cite{lpcivita,shamseddine2012,berz+shamseddine2003}.
In analogy with the discussion in Section \ref{section comparison with LC}, we prove coherence with the existing non-Archimedean integral. 

An application of the real-valued integral is presented in Section \ref{sec applications}, where we discuss the representation of some real distributions as pointwise functions defined on the Levi-civita field, sharpening some of the results obtained in \cite{lpcivita}.

\subsection{Preliminary definitions}

Throughout the paper $(\F,<_\F)$ will denote a non-Archimedean field extension of $(\R,<_\R)$.
In particular, we will suppose that $\R \subset \F$ and that for every $x, y \in \R$ $x <_\R y$ if and only if $x <_\F y$.
Due to this assumption, we will often write $(\F,<)$ instead of $(\F,<_\F)$.

A number $x \in \F$ is called
\begin{itemize}
	\item \emph{infinitesimal} if $|x| \leq r$ for every $r \in \R$, $r > 0$;
	\item \emph{finite} if there exists $r \in \R$ such that $|x|<r$;
	\item \emph{appreciable} if $x$ is finite and non-infinitesimal;
	\item \emph{infinite} if $|x|>r$ for every $r \in\R$.
\end{itemize}

If $x \in \F$ is infinitesimal, we will write $x \simeq 0$.
If $x \in \F$ is a nonzero infinitesimal, we will write $x \sim 0$.
In analogy with Robinson's framework of analysis with infinitesimals, if $x \in \F$, we will refer to the set $\mu(x)=\{ y \in \F : |x-y|\simeq 0 \}$ as the \emph{monad} of the point $x$. 
Recall also that monads are not intervals \cite{lpcivita}.

We define $\F_{fin} = \{x\in\F : \exists r \in \R : |x|<r\}$, i.e. $\F_{fin}$ is the ring of all finite elements of $\F$.

We find it also useful to define the \emph{standard part} of an element of $\F$.

\begin{definition}\label{def sh}
	if $x \in \F$, 
	we define
	$$
	\sh{x} = 
	\left\{
	\begin{array}{ll}
	\inf\{y\in\R : x\leq y\} = \sup\{z\in\R : z\leq x\} & \text{if } |x|<r \text{ for some } r \in\R\\
	+\infty & \text{if } x>r \text{ for all } r \in \R\\
	-\infty & \text{if } x<r \text{ for all } r \in \R.
	\end{array}
	\right.
	$$
\end{definition}

The function $\sh : \F \rightarrow \R \cup \{+\infty,-\infty\}$ is well-defined and surjective.
Moreover, it is a homomorphism between the rings $\F_{fin}$ and $\R$.

\begin{lemma}
	For every $x, y \in \F_{fin}$
	\begin{enumerate}
		\item $\sh{(x+y)}=\sh{x}+\sh{y}$;
		\item $\sh{(xy)}=\sh{x}\sh{y}$.
	\end{enumerate}
\end{lemma}
\begin{proof}
	Let $x = \sh{x}+\varepsilon_x$ and $y = \sh{y}+\varepsilon_y$, with $\varepsilon_x$ and $\varepsilon_y$ infinitesimals in $\F$.
	Then $x+y = \sh{x}+\sh{y}+\varepsilon_x+\varepsilon_y$.
	Since $\sh{x}+\sh{y} \in \R$ and $\varepsilon_x+\varepsilon_y$ is a sum of two infinitesimals, $\sh{(x+y)}=\sh{x}+\sh{y}$.
	
	Similarly, $xy = \sh{x}\sh{y}+\varepsilon_x \sh{y} + \varepsilon_y \sh{x} + \varepsilon_x \varepsilon_y$, and
	\begin{itemize}
		\item $\sh{x}\sh{y} \in\R$;
		\item $\varepsilon_x \sh{y}$, $\varepsilon_y \sh{x}$ and $\varepsilon_x \varepsilon_y$ are infinitesimals.
	\end{itemize}
	We deduce that $\sh{(xy)}=\sh{x}\sh{y}$.
\end{proof}

Another useful notion borrowed from Robinson's framework is that of \emph{nearstandard point} in a set.

\begin{definition}\label{def nearstandard}
	Let $A \subseteq \F$.	
	We will say that a point $x \in A$ is nearstandard in $A$ iff $\sh{x} \in A$.
\end{definition}

For every $a, b \in \mathbb{F}$ with $a \leq b$ we will denote by $[a,b]_{\mathbb{F}}$ the set $\{ x \in \mathbb{F} : a \leq x \leq b \}$, and by $(a,b)_{\mathbb{F}}$ the set $\{ x \in \mathbb{F} : a < x < b \}$.
The sets $[a,b)_{\mathbb{F}}$ and $(a,b]_{\mathbb{F}}$ are defined accordingly.
The above definitions are extended in the usual way if $a = -\infty$ or $b = +\infty$. 
If $\mathbb{F}=\R$, we will often write $[a,b]$ instead of $[a,b]_{\R}$.

For all $a, b \in \F$ with $a<b$, we will denote by $I(a,b)$ any of the sets $(a,b)_\F$, $[a,b)_\F$, $(a,b]_\F$ or $[a,b]_\F$.
We will call such sets \emph{bounded intervals} of $\F$.
The length of an interval of the form $I(a,b)$ is denoted by $l(I(a,b))$ and is defined as $b-a$.


Finally, we will denote by $\leb^n$ the Lebesgue measure over $\R^n$.

\section{A real-valued measure on non-Archimedean extensions of $\R$}\label{sec loeb}

We begin our treatment of a real-valued measure on non-Archimedean extensions of $\R$ by introducing an outer measure over $\F$ that assumes values in the extended real numbers $\R \cup \{+\infty\}$.
This outer measure is obtained from the standard part of the length of an interval, in analogy with the Lebesgue outer measure.

\begin{definition}\label{def outer}
	For every $a, b \in \F$, $a\leq b$, define $l_L(I(a,b)) = \sh{(b-a)}$.
	For every $A \subseteq \F$ such that there exists a sequence of bounded intervals $\{I_n\}_{n\in\N}$ satisfying $A \subseteq \bigcup_{n \in \N} I_n$,
	define
	$$
		\extl(A) = \inf\left\{\sum_{n \in \N}l_L(I_n) : A \subseteq \bigcup_{n\in\N} I_n \right\}.
	$$
	If for every sequence of bounded intervals $\{I_n\}_{n\in\N}$ we have $A \not\subseteq \bigcup_{n \in \N} I_n$, define $\extl(A) = +\infty$.
\end{definition}

The last condition of Definition \ref{def outer} ensures that $\extl$ is defined on the powerset of $\F$, since e.g. $\F$ itself might not be contained in the union of any countable union of bounded intervals.
This property is essential in proving that $\extl$ is an outer measure over $\F$.

\begin{lemma}
	The function $\extl : \P(\F)\rightarrow\R\cup\{+\infty\}$ is an outer measure.
\end{lemma}
\begin{proof}
	We have already observed that $\extl$ is defined on $\P(\F)$.
	
	Since $l_L(I) \geq 0$ for every interval $I \subseteq \F$, $\extl(A) \geq 0$ for all $A \subseteq \F$.
	Moreover, $\extl(\emptyset)=0$.
	
	In order to prove monotonicity, i.e.\ that $\extl(A) \leq \extl(B)$ whenever $A \subseteq B$, notice that if $B \subseteq \bigcup_{n \in \N} I_n$, then also $A \subseteq \bigcup_{n \in \N} I_n$.
	As a consequence we get
	$$
	\extl(A) =
	\inf\left\{\sum_{n \in \N}l_L(I_n) : A \subseteq \bigcup_{n \in \N} I_n \right\}
	\leq
	\inf\left\{\sum_{n \in \N}l_L(I_n) : B \subseteq \bigcup_{n \in \N} I_n \right\},
	$$
	as desired.
	
	Finally, we need to prove $\sigma$-subadditivity of $\extl$, i.e. that if $A_n \subseteq \F$ for all $n\in\N$, then
	$$
	\extl\left(\bigcup_{n \in \N} A_n\right) \leq \sum_{n \in \N}\extl(A_n).
	$$
	The result is trivially true if $\sum_{n \in \N}\extl(A_n)=+\infty$, so assume that this is not the case.
	
	Suppose then that  $\sum_{n \in \N}\extl(A_n) \in \R$: this entails also $\extl(A_n)\in\R$ for every $n\in\N$.
	Then for every $\varepsilon \in \R$, $\varepsilon >0$, there exists a family of sets $\{I_{n,k}^\varepsilon\}_{k\in\N}$ such that
	\begin{itemize}
		\item $A_n \subseteq \bigcup_{k \in \N} I_{n,k}^\varepsilon$ and
		\item $\sum_{k \in \N}l_L(I_{n,k}^\varepsilon) \geq \extl(A_n)+\frac{\varepsilon}{2^n}$.
	\end{itemize}
	We have also the inclusion $\bigcup_{n\in\N} A_n \subseteq \bigcup_{n \in \N}\left( \bigcup_{k\in\N} I_{n,k}^\varepsilon \right)$.
	From monotonicity of the outer measure, we obtain
	$$
		\extl\left(\bigcup_{n \in \N} A_n\right)
		\leq
		\sum_{n\in\N} \left( \sum_{k \in \N}l_L(I_{n,k}^\varepsilon) \right)
		\leq
		\sum_{n \in \N} \left(\extl(A_n)+\frac{\varepsilon}{2^n}\right)
		=
		\varepsilon+\sum_{n \in \N} \extl(A_n).
	$$
	By the arbitrariness of the real parameter $\varepsilon>0$, we conclude that $\extl$ is $\sigma$-subadditive.
\end{proof}

\begin{remark}
	From monotonicity of the outer measure $\extl$ we deduce that every set contained in an interval of an infinitesimal length has outer measure $0$, while if a set contains intervals of length at least $n$ for every $n \in\N$, then its outer measure is infinite. 
	As a consequence, 
	$\extl(\F_{fin})=+\infty$ and $\extl(A) = +\infty$ whenever $A \supset \F_{fin}$.
\end{remark}

From the outer measure $\extl$ defined over $\P(\F)$, it is possible to obtain a $\sigma$-algebra of measurable sets.

\begin{definition}
	Given the outer measure $\extl$ on $\F$, the following family of subsets of $\F$ is
	called the \emph{Caratheodory $\sigma$-algebra} associated to $\extl$:
	$$
	\mathfrak{C}\ =\
	\left\{A\subseteq\F:
	\extl(B)=\extl(B\cap A)+\extl(B\setminus A)
	\text{ for all }B\subseteq\F\right\}.$$
	If $A \in \mathfrak{C}$, we will say that $A$ is $L$-measurable.
\end{definition}

A well known theorem of Caratheodory states that the above family $\mathfrak{C}$ is indeed a
$\sigma$-algebra, and that the restriction of $\extl$ to $\mathfrak{C}$, that we will denote by $\loeb$, is a complete measure, i.e. a measure such that $\loeb(A)=0$ implies that, for every $B\subseteq A$, $B \in \mathfrak{C}$ and $\loeb(B)=0$ (for more details on the Caratheodory's extension theorem and on complete measures, we refer for instance to Chapter 2 of \cite{Yeh}).

From completeness of $\mathfrak{C}$, we obtain the following regularity property.
In the sequel, we will use it as a criterion for $L$-measurability.

\begin{lemma}\label{lemma A c B c C}
	Let $A, C \in \mathfrak{C}$.
	If $\loeb(A) = \loeb(C)<+\infty$, then for every $B \subseteq \F$ that satisfies $A \subseteq B \subseteq C$, $B\in\mathfrak{C}$ and $\loeb(B) = \loeb(A) = \loeb(C)$.
\end{lemma}
\begin{proof}
	Let $A, B$ and $C$ satisfy the hypotheses of the lemma.
	By monotonicity of the outer measure, $\loeb(A) = \extl(A) \leq \extl(B) \leq \extl(C) = \loeb(C)$.
	Thus $\extl(B) = \loeb(A) = \loeb(B)$.
	
	Since $A$ is measurable, $\extl(B) = \extl(A \cap B) + \extl(B \setminus A)$.
	However, $A \cap B = A$ and $\extl(B) = \extl(A)$, so that $\extl(B\setminus A) = 0$.
	Since $\mathfrak{C}$ is complete, $B \setminus A \in \mathfrak{C}$.
	Thus $B = A \cup \left(B\setminus A\right)$, i.e.\ $B$ is the union of two $L$-measurable sets. Since $\mathfrak{C}$ is a $\sigma$-algebra, hence closed also for finite unions, $B$ is also $L$-measurable.
\end{proof}

The measure $\loeb$ shares some properties with the Lebesgue measure.
For instance, it is translation invariant.

\begin{lemma}
	If $A \subseteq \F$ is $L$-measurable, then for every $x \in \F$ the set
	$$
		A+x = \{y : \exists a \in A : y=a+x\}
	$$
	is $L$-measurable and $\loeb(A) = \loeb(A+x)$
\end{lemma}
\begin{proof}
	This is consequence of the two properties
	$$
		A \subseteq \bigcup_{n \in \N} I_n \Rightarrow A+x \subseteq \bigcup_{n \in \N} (I_n+x)
	$$
	and $l_L(I)=l_L(I+x)$ for every interval $I \subseteq \F$ and for every $x \in \F$.
	
	The proof can then be carried out as in the usual proof of translation invariance of the Lebesgue measure; for more details we refer e.g.\ to Lemma 3.15 and Theorem 3.16 of \cite{Yeh}.
\end{proof}

Notice however that $\loeb$ is not positively homogeneous, and that the very same notion of positive homogeneity needs to be adapted to the non-Archimedean setting.

\begin{proposition}\label{prop positive homogeneity}
	If $A \subseteq \F$ is $L$-measurable and if $\loeb(A) <+\infty$, then for every $x \in \F_{fin}$ the set $$xA= \{y : \exists a \in A : y=a+x\}$$ is $L$-measurable and $\loeb(xA) = |\sh{x}|\loeb(A)$.
\end{proposition}
\begin{proof}
	If $x=0$, the desired result is trivially satisfied, since $0A = \{0\}$.
	
	For every $x \in \F_{fin}$, $x\ne 0$, and for every $A \subseteq \F$ with $\extl(A) <+\infty$, we have the inclusion
	$$
	A \subseteq \bigcup_{n \in \N} I_n \Rightarrow xA \subseteq \bigcup_{n \in \N} (xI_n)
	$$
	and the equality $l_L(xI)=|\sh{x}|l_L(I)$ for every interval $I \subseteq \F$.
	This is sufficient to conclude $\extl(xA)\leq |\sh{x}|\extl(A)$ for every $A \subseteq \F$ with $\extl(A) <+\infty$ and for every $x \in \F_{fin}$.
	
	If $x$ is an infinitesimal, $\extl(xA)\leq |\sh{x}|\extl(A) = 0 \cdot \extl(A) = 0$. 
	If $x \in \F_{fin}$ is appreciable, then $x^{-1}$ is neither infinite nor infinitesimal.
	Then also
	$$\extl(A) = \extl(x^{-1} xA) \leq |\sh{(x^{-1})}| \extl(xA) \leq |\sh{(x^{-1})}||\sh{x}|\extl(A) = \extl(A).$$
	By combining these results, we obtain that $\extl{(xA)}=|\sh{x}|\extl(A)$ for every $A \subseteq \F$ with $\extl(A) <+\infty$ and for every $x \in \F_{fin}$.
	
	The proof of $L$-measurability of the set $xA$ under the hypothesis that $A$ is $L$-measurable can be obtained with an argument analogous to that of Theorem 3.18 of \cite{Yeh}.
\end{proof}

\begin{example}
	If $A$ is a $L$-measurable set of an infinite measure, then positive homogeneity fails, since for $x \simeq 0$ it would lead to the indeterminate form $\loeb(xA) = |\sh{x}|\loeb(A) = 0 \cdot +\infty$.
	In fact, let $a \in \F$ be a positive infinite number and consider the set $A = [0,a]$.
	Then $\loeb(xA)$ can be either zero, any positive real number, or $+\infty$, depending upon the value of $x$.
	For instance, for every $r \in \F_{fin}$ $\loeb(ra^{-1}A)=\sh{r}$, $\loeb(ra^{-2}A)=0$ and $\loeb(ra^{-1/2}A)=+\infty$.
\end{example}

\begin{remark}
	The measure $\loeb$ shares many properties with the Lebesgue measure over $\R$.
	For instance, it is uniform, positively homogeneous and translation invariant over $\F$.
	Despite these similarities, $\loeb$ does not satisfy other relevant properties of the Lebesgue measure: for instance, it is not $\sigma$-finite, since $\F$ is not the union of countably many sets of a finite measure. Notice however that the restriction of $\loeb$ to $\F_{fin}$ is $\sigma$-finite.
	In addition, the complement of a null set needs not be a dense subset of $\F$ or $\F_{fin}$ (compare this property with the one discussed in Observation 3.7 of \cite{Yeh}), since $\F$ and $\F_{fin}$ are totally disconnected with respect to the topology induced by the metric \cite{payne}.
\end{remark}

\begin{remark}
	The outer measure $\extl$ and the corresponding measure $\loeb$ can be suitably rescaled.
	E.g. if one is interested in working with a measure that assigns length $1$ to the intervals of the form $I(0,\varepsilon)$, with either $\varepsilon \ll 1$ or $\varepsilon \gg 1$, then it is possible to assume the alternative definition $l_{\varepsilon}(I(a,b))=\sh{\left(\frac{b-a}{\varepsilon}\right)}$.
	The resulting measure
	$$
	{}_\varepsilon\extl(A) = \inf\left\{\sum_{n \in \N}l_{\varepsilon}(I_n) : A \subseteq \bigcup_{n\in\N} I_n \right\}
	$$
	would have the desired property ${}_\varepsilon\extl(I(0,\varepsilon))=1$.
	Consequently, the family of sets with a finite ${}_\varepsilon\extl$ outer measure can be interpreted as the family of sets whose measure is of the same magnitude as that of the intervals $I(0,\varepsilon)$.
	
	Many of the properties already proved for $\extl$ and $\loeb$, such as translation invariance and positive homogeneity as described in Proposition \ref{prop positive homogeneity}, are still valid for these rescaled measures.
\end{remark}

\subsection{Relation with the Lebesgue measure over $\R$}

We will now study the relation between the measure $\loeb$ over $\F$ and the Lebesgue measure over $\R$.
	Notice that Lebesgue measurable subsets of $\R$ are not in general $\loeb$-measurable in $\F$.
	Consider for instance the real intervals $A=[0,1]_\R$ and $B=[0,1]_\F$.
	Since
	$$1 = \extl([0,1]_\F) \ne \extl([0,1]_\R) + \extl([0,1]_\F \setminus [0,1]_\R) = 2,$$
	we conclude that $[0,1]_\R \not \in \mathfrak{C}$.

However, Lebesgue measurable sets over $\R$ have a canonical $L$-measurable representative in $\F$.
Moreover, the measure of this representative is equal to the Lebesgue measure of the original set.

We will prove this result at first by showing that, for every Lebesgue measurable set $A \subseteq \R$, the set $\st^{-1}(A)=\{x \in \F : \sh{x}\in A\}$ has outer measure equal to $\leb(A)$.
Then we will prove that if $A \subseteq \R$ is Lebesgue measurable, then $\st^{-1}(A)\in\mathfrak{C}$.
Notice that $\st^{-1}(A) \subseteq \F_{fin}$, since if $x \not \in \F_{fin}$, $\st(x) = \pm \infty \not \in A$ for every $A \subseteq \R$.

\begin{proposition}\label{prop coerenza outer measure e misura di lebesgue} 
	If $A \subseteq \R$ is Lebesgue measurable, then 
	${\leb(A)} = \extl(\st^{-1}(A))$.
\end{proposition}
\begin{proof}
	Consider an interval $[a,b]_\R$
	and, for all $n\in\N$, define the intervals $I(a-1/n, b+1/n) \subset \F$.
	We have $\st^{-1}([a,b]_\R) \subset I(a-1/n, b+1/n)$ and
	$$\extl(\st^{-1}([a,b]_\R)) \leq \loeb\left(I\left(a-\frac{1}{n}, b+\frac{1}{n}\right)\right) = b-a+\frac{2}{n}.$$
	Since $\inf_{n\in\N}\left\{ b-a+\frac{2}{n} \right\} = b-a$, $\extl(\st^{-1}(I(a,b))) \leq b-a$.
	Notice also that $I(a,b) \subset \st^{-1}([a,b]_\R)$, $I(a,b) \in \mathfrak{C}$ and $\loeb(I(a,b))=b-a$.
	Then, by monotonicity of the outer measure, $\extl(\st^{-1}([a,b]_\R))=b-a$.
	
	Consider now an arbitrary Lebesgue measurable set $A \subseteq \R$.	
	Recall that its Lebesgue measure $\leb(A)$ can be defined as
	$$
		\leb(A) = \inf\left\{\sum_{n \in \N}l(I_n) : A \subseteq \bigcup_{n\in\N} I_n \right\}.
	$$
	In the above formula, $I_n \subseteq \R$ for all $n \in\N$ and $l(I_n)$ is the usual length of the real interval $I_n$.
	%
	By definition, we have also
	$$
		\extl(\st^{-1}(A)) = \inf\left\{\sum_{n \in \N}l_L(J_n) : \st^{-1}(A) \subseteq \bigcup_{n\in\N} J_n \right\}.
	$$
	Consider now the real intervals $I_n = \{\sh{x} : x \in J_n\}$, and notice that if $\st^{-1}(A) \subseteq \bigcup_{n \in \N} J_n$, then $A \subseteq \bigcup_{n\in\N} I_n$.
	By definition of $l_L$, we have also
	$l_L(J_n) = l(I_n)
	$,
	so that
	$$
	\left\{\sum_{n \in \N}l_L(J_n) : \st^{-1}(A) \subseteq \bigcup_{n\in\N} J_n \right\}
	\subseteq
	\left\{\sum_{n \in \N}l(I_n) : A \subseteq \bigcup_{n\in\N} I_n \right\}.
	$$
	This inclusion entails the inequality $\leb(A)\leq \extl(\st^{-1}(A))$.
	
	In order to prove that the opposite inequality is also true, let $\varepsilon \in\R$, $\varepsilon >0$ and let $\{I_n\}_{n\in\N}$ satisfy
	\begin{itemize}
		\item $A \subseteq \bigcup_{n\in\N} I_n \subseteq \R$ and
		\item $\sum_{n \in \N}l(I_n) \leq \leb(A)+\varepsilon$.
	\end{itemize} 
	Then $\st^{-1}(A) \subseteq \bigcup_{n\in\N} \st^{-1}(I_n)$.
	By $\sigma$-subadditivty and monotonicity of $\extl$, we deduce
	$$
		\extl(\st^{-1}(A))
		\leq
		\extl\left(\bigcup_{n\in\N}\st^{-1}(I_n)\right)
		\leq
		\sum_{n \in\N} \extl(\st^{-1}(I_n)).
	$$
 	In the first part of the proof we have shown that $\loeb(\st^{-1}(I)) = l(I)$ for every real interval $I$.
	From this equality we deduce
	$\extl(\st^{-1}(A)) \leq \sum_{n \in \N}l(I_n) \leq \leb(A)+\varepsilon$.
	By the arbitrariness of the real parameter $\varepsilon$, we obtain $\extl(\st^{-1}(A)) \leq \leb(A)$, as desired.
\end{proof}

\begin{proposition}\label{prop A lebesgue implica st-1(A) loeb} 
	If $A \subseteq \R$ is Lebesgue measurable, then $\st^{-1}(A)
	\in\mathfrak{C}$.
\end{proposition}
\begin{proof}
	By $\sigma$-additivity of the measure $\loeb$, it is sufficient to prove that for every bounded Lebesgue measurable set $A \subseteq\R$, $\st^{-1}(A)\in\mathfrak{C}$.
	Once we have proven this result, the fact that $\st^{-1}(A)\in\mathfrak{C}$ for every Lebesgue measurable $A \subseteq \R$ can be obtained by the fact that $\mathfrak{C}$ is closed under countable unions.
	For this reason, in the sequel of the proof we will suppose that $A$ is bounded.
	
	By definition of the outer measure $\extl$ and by Theorem 2.24 of \cite{Yeh}, it is sufficient to prove that, if $A\subseteq\R$ is Lebesgue measurable, then
	\begin{equation}\label{eqn measurable}
		\sh{(b-a)}=\loeb(I(a,b))=\extl(\st^{-1}(A)\cap I(a,b))+\extl(I(a,b)\setminus \st^{-1}(A))
	\end{equation}
	for every $I(a,b) \subseteq \F$ with $a, b \in \F$, $a \leq b$.
	
	Recall also that, by subadditivity of the outer measure $\extl$, the inequality
	$$
		\sh{(b-a)}\leq \extl(\st^{-1}(A)\cap I(a,b))+\extl(I(a,b)\setminus \st^{-1}(A))
	$$
	is always satisfied, so we only need to prove the opposite inequality under the additional hypothesis that $\sh{(b-a)}<+\infty$.
	
	Notice that $[a,b]_\R$ is an interval, so it is Lebesgue measurable.
	The hypothesis that $A$ is Lebesgue measurable ensures then that $A \cap [a,b]_\R$ and $[a,b]_\R \setminus A$ are Lebesgue measurable subsets of $\R$.
	Applying Proposition \ref{prop coerenza outer measure e misura di lebesgue} we obtain
	$$
	\extl(\st^{-1}(A\cap [a,b]_\R) = \leb(A \cap [a,b]_\R)
	$$
	and
	$$
	\extl(\st^{-1}([a,b]_\R \setminus A)) = \leb([a,b]_\R \setminus A).
	$$
	Since $\st^{-1}(A)\cap I(a,b) \subseteq \st^{-1}(A\cap  [a,b]_\R)$ and $I(a,b)\setminus \st^{-1}(A) \subseteq \st^{-1}([a,b]_\R \setminus A)$, by monotonicity of the outer measure we have
	$$
	\extl(\st^{-1}(A)\cap I(a,b))\leq\extl(\st^{-1}(A\cap [a,b]_\R)) = \leb(A \cap [a,b]_\R)
	$$
	and
	$$
	\extl(I(a,b)\setminus \st^{-1}(A))\leq \extl(\st^{-1}([a,b]_\R \setminus A)) = \leb([a,b]_\R \setminus A).
	$$
	Putting together the two inequalities, we obtain
	$$
		\extl(\st^{-1}(A)\cap I(a,b))+\extl(I(a,b)\setminus \st^{-1}(A))
		\leq
		\leb(A \cap [a,b]_\R) + \leb([a,b]_\R \setminus A) = \leb([a,b]_\R)=\sh(b-a).
	$$
	
	The above inequality is sufficient to conclude that equality \eqref{eqn measurable} is satisfied for every $I(a,b) \subseteq \F$ with $a, b \in \F$, $a \leq b$.
	As we argued in the beginning of the proof, this is sufficient to entail that $\st^{-1}(A)\in\mathfrak{C}$, as desired.
\end{proof}

\begin{theorem}\label{thm coerenza L-meas e misura di lebesgue} 
	If $A \subseteq \R$ is Lebesgue measurable, then the set $\st^{-1}(A)=\{x \in \F : \sh{x}\in A\}$ is $L$-measurable, and $\leb(A) = \loeb(\st^{-1}(A))$.
\end{theorem}
\begin{proof}
	By Proposition \ref{prop A lebesgue implica st-1(A) loeb}, if $A\subseteq \R$ is Lebesgue measurable, then $\st^{-1}(A)\in\mathfrak{C}$.
	As a consequence, $\extl(\st^{-1}(A))=\loeb(\st^{-1}(A))$.
	By Proposition \ref{prop coerenza outer measure e misura di lebesgue}, $\extl(\st^{-1}(A))=\leb(A)$, so that also $\loeb(\st^{-1}(A))=\leb(A)$.
\end{proof}

Conversely, a $L$-measurable subset of $\F_{fin}$ corresponds via the standard part function to a Lebesgue measurable subset of $\R$. In other words, the standard part function is measure-preserving.

\begin{theorem}\label{thm coerenza L-meas e misura di lebesgue2}
	If $A \subseteq \F_{fin}$ is $L$-measurable, then the set $\sh A=\{\sh{x} \in \R : x\in A\}$ is Lebesgue measurable, and $\loeb(A) = \leb(\sh{A})$.
\end{theorem}
\begin{proof}
	Notice that for every $A \subseteq \F_{fin}$, if 
	$
		A \subseteq \bigcup_{n \in \N} I_n,
	$
	then
	$
		\sh{A} \subseteq \bigcup_{n \in \N} \sh{I_n},
	$
	and $\sh{I_n}$ is a closed interval or a singleton for every $n\in\N$.
	Consequently, if we denote by $\overline{\leb}$ the Lebesgue outer measure over $\R$, $\overline{\leb}(\sh{A}) \leq \loeb(A)$.
	
	As in the proof of Proposition \ref{prop A lebesgue implica st-1(A) loeb}, we will consider at first only sets included in an interval of a finite length. The desired result for arbitrary $L$-measurable sets can then be obtained by $\sigma$-additivity of the measures $\loeb$ and $\leb$.
	
	If $A \subseteq [-n,n]_{\F}$ is $L$-measurable, then we have
	$$
		\overline{\leb}(\sh{A}) + \overline{\leb}(\sh{A^c}) \leq \loeb(A)+ \loeb(A^c) = 2n.
	$$
	However, countable subadditivity of the outer measure $\overline{\leb}$ implies that
	$$
		\overline{\leb}([-n,n]_{\R}) \leq \overline{\leb}(\sh{A}) + \overline{\leb}(\sh{A^c}).
	$$
	Putting together both inequalities, we conclude that $\overline{\leb}(\sh{A}) + \overline{\leb}(\sh{A^c})=2n$ and, taking into account that $\overline{\leb}(\sh{B}) \leq \loeb(B)$ for every $B \subseteq \F_{fin}$, we conclude that
	$\overline{\leb}(\sh{A}) = \loeb(A)$ for every $L$-measurable set $A \subseteq [-n,n]$.
	
	In order to prove that, if $A \subseteq [-n,n]_{\F}$ is $L$-measurable, then it is also Lebesgue measurable, we will prove that
	$$
		b-a = \overline{\leb}(I(a,b)) = \overline{\leb}(\sh{A}\cap I(a,b)) + \overline{\leb}(I(a,b)\setminus \sh{A})
	$$
	for every $I(a,b)\subseteq \R$, with $a, b\in\R$, $a\leq b$.
	Since $A$ is $L$-measurable, it satisfies the Caratheodory condition
	$$
		\sh(b-a) = \loeb(A\cap I(a,b)) + \loeb(I(a,b)\setminus A)
	$$
	for every $I(a,b) \subseteq \F$ with $a, b \in \F$, $a \leq b$.
	Notice also that $\sh{\left(A \cap I(a,b) \right)} = \sh{A} \cap [\sh{a},\sh{b}]_{\R}$
	and $\sh{\left( I(a,b)\setminus A \right)} 
	\supseteq [\sh{a},\sh{b}] \setminus \sh{A}$.
	By the previous part of the proof, if $a, b \in \R$ we have
	$$
		\overline{\leb}(\sh{A}\cap [a,b]_\R) = \loeb(A \cap [a,b]_{\F}) \text{ and } \overline{\leb}([a,b]_{\R} \setminus \sh{A}) \leq \loeb([a,b]_{\F} \setminus A).
	$$
	Taking into account that $A$ satisfies the Caratheodory measurability condition over $\F$, we obtain
	$$
		\overline{\leb}(\sh{A}\cap I(a,b)) + \overline{\leb}(I(a,b)\setminus \sh{A})
		\leq
		\loeb(A \cap [a,b]_{\F}) + \loeb([a,b]_{\F} \setminus A)
		= b-a,
	$$
	as desired.
	
	Thus we have proved that for every bounded $L$-measurable set $A \subseteq \F_{fin}$, $\sh{A}$ is Lebesgue measurable and $\loeb(A) = \leb(\sh{A})$.
	For an arbitrary $L$-measurable set $A \subseteq \F_{fin}$, we have already argued that the desired result can be obtained from $\sigma$-additivity of the measures $\loeb$ and $\leb$.
\end{proof}

\subsection{Relation with the Loeb measure on hyperreal fields}
If $\F$ is a {sufficiently saturated} field of hyperreal numbers of Robinson's framework of analysis with infinitesimals, so that $\leb_L$, the Loeb measure associated to the real Lebesgue measure, can be defined, then it is possible to study relation between $\loeb$ and ${\leb}_L$.
From Theorem \ref{thm coerenza L-meas e misura di lebesgue}, we can already conclude that both measures agree on the preimage of Lebesgue measurable subsets of $\R$ via the standard map function.

However the two measures are different: consider for instance an infinite hypernatural number $N$ and the set $A = \bigcup_{n=1}^N [n,n+N^{-1}]_{\hR}$.
Then $\leb_L(A) = N \cdot N^{-1} = 1$.
Notice that $A \not \subseteq \hR_{fin}$ and that $A$ is a hyperfinite union of intervals of an infinitesimal length.
Recall that hyperfinite subsets in Robinson's framework of analysis with infinitesimals have uncountable external cardinality (while, by definition, they have finite internal cardinality, since they can be put in an internal bijection with an internal initial segment of $\ns{\N}$. For more details on the distinction between internal and external cardinality, we refer to \cite{go}).
As a consequence, every countable sequence of intervals $\{I_n\}_{n\in\N}$ satisfying $A \subseteq \bigcup_{n \in \N} I_n$ must include at least one interval of an infinite length, so that $\extl(A) = +\infty$.
We deduce that either $A \not \in \mathfrak{C}$ or $\loeb(A)=+\infty$.

Despite these differences, the measure $\loeb$ is compatible with the Loeb measure ${\leb}_L$ over $\hR_{fin}$ in the sense that if a subset of $\hR_{fin}$ is $L$-measurable, then it is also $\leb_L$-measurable and the two measures coincide.


\begin{theorem}
	Let $F=\hR$, a {sufficiently saturated} field of hyperreal numbers, and denote by $\mathfrak{L}$ the $\sigma$-algebra of Loeb measurable subsets of $\hR$.
	Then for every $A \subseteq \hR_{fin}$, if $A \in \mathfrak{C}$ then also $A \in \mathfrak{L}$. Moreover, $\loeb(A) = \leb_L(A)$.	
\end{theorem}
\begin{proof}
	Recall that, by the Caratheodory's extension theorem,
	$\mathfrak{C}$ is the smallest $\sigma$-algebra containing the family
	$$
	\mathcal{C}=\left\{ \bigcup_{n \in \N} I_n : \{I_n\}_{n\in\N} \text{ is a sequence of pairwise disjoint bounded intervals of } \hR \right\}.
	$$
	
	By definition of $l_L$, for every interval $I\subseteq \hR$ we have $l_L(I)=\leb_L(I)$, so that for every sequence of bounded intervals $\{I_n\}_{n\in\N}$ we have the equality
	$$
	\sum_{n \in \N}l_L(I_n) = \sum_{n\in\N} \leb_L(I_n).
	$$
	Moreover, $\bigcup_{n \in \N} I_n \in \mathfrak{L}$ and
	\begin{equation}\label{eqn coherence loeb 1}
		\leb_L\left(\bigcup_{n \in \N}I_n \right) = \sum_{n \in \N}l_L(I_n) = \loeb\left(\bigcup_{n \in \N}I_n \right).
	\end{equation}
	
	For every $n \in\N$, define now the following families of subsets of $\hR_{fin}$.
	\begin{itemize}
		\item $\mathcal{C}_{n}=\mathcal{C} \cap \P([-n,n]_{\hR})$;
		\item $\mathcal{C}_{fin}=\mathcal{C} \cap \P(\hR_{fin})$;
		\item $\mathfrak{C}_{n} = \mathfrak{C} \cap \P([-n,n]_{\hR})$;
		\item $\mathfrak{C}_{fin} = \mathfrak{C} \cap \P(\hR_{fin})$;
		\item $\mathcal{E}_{n} = \left\{ A \in \mathfrak{C} \cap \mathcal{C} : A \subseteq [-n,n]_{\hR} \text{ and } \leb_L(A) = \loeb(A) \right\}$; and
		\item $\mathcal{E}_{fin} = \left\{ A \in \mathfrak{C} \cap \mathcal{C} : A \subseteq {\hR}_{fin} \text{ and } \leb_L(A) = \loeb(A) \right\}$.
	\end{itemize}
	Notice that we are not assuming that the members of any of the above families must be internal.
	By equation \eqref{eqn coherence loeb 1}, $\leb_L$ and $\loeb$ assume the same values on elements of $\mathcal{C}_{fin}$. In addition, we have the inclusions $\mathcal{C}_{n} \subseteq \mathcal{E}_{n}$ for every $n\in\N$.

	In order to prove that for every
	$A \in \mathfrak{C}_{fin}$ then also $A \in \mathcal{E}_{fin}$, i.e.\ that $A$ is Loeb measurable and $\loeb(A) = \leb_L(A)$,
	we will prove that $\mathfrak{C}_{n} \subseteq \mathcal{E}_{n}$ for all $n \in\N$ by using Dinkyn's $\pi$--$\lambda$ theorem.
	The desired result can then be obtained by noticing that, by $\sigma$-additivity,
	\begin{eqnarray*}
		\leb_L(A) &=& \displaystyle \sum_{n \in \N} \leb_L\left( A \cap \left([-n-1,-n] \cup [n,n+1] \right) \right) \text{ and}\\
		\loeb(A) &=& \displaystyle \sum_{n \in \N} \loeb\left( A \cap \left([-n-1,-n] \cup [n,n+1] \right) \right),
	\end{eqnarray*}
	and that the inclusion $\mathcal{C}_{n} \subseteq \mathcal{E}_{n}$ for all $n \in\N$ entails that
	$$
		\leb_L\left( A \cap \left([-n-1,-n] \cup [n,n+1] \right) \right) = \loeb\left( A \cap \left([-n-1,-n] \cup [n,n+1] \right) \right)
	$$
	for all $n \in\N$, so that also $\leb_L(A) = \loeb(A)$.

	Recall that a $\pi$-system over a set $X$ is a family of subsets of $X$ closed under finite intersections, and a $\lambda$-system over $X$ is a family of subsets of $X$ that
	\begin{enumerate}
		\item contains the empty set;
		\item is closed under complements;
		\item is closed under countable disjoint unions.
	\end{enumerate}
	It is easy to see that $\mathcal{C}_{n}$ is a $\pi$-system for every $n \in\N$.
	We now want to prove that $\mathcal{E}_n = \left\{ A \subseteq \hR_{fin} : \leb_L(A) = \loeb(A) \right\}$
	is a $\lambda$-system for every $n \in\N$. 
	\begin{enumerate}
		\item Clearly $\emptyset \in \mathcal{E}_{n}$, since $\emptyset \subseteq [-n,n]_{\hR}$ and $\leb_L(\emptyset)=\loeb(\emptyset)=0$.
		\item Suppose now that $\leb_L(A) = \loeb(A)$ for some $A \in \mathfrak{C} \cap \mathcal{C}$,  $A \subseteq [-n,n]_{\hR}$, and let $A^c = [-n,n]_{\hR} \setminus A$. Taking into account that $[-n,n] \in \mathfrak{C} \cap \mathcal{C}$, that $\leb_L([-n,n]_{\hR})=\loeb([-n,n]_{\hR}) = 2n$ for every $n \in \N$ and that $A$ and $A^c$ are disjoint, $\leb_L(A)+\leb_L(A^c)=\loeb(A)+\loeb(A^c)=2n$. Since we have assumed that $\leb_L(A)=\loeb(A)$, we have also $\leb_L(A^c)=\loeb(A^c)$, i.e.\ $A^c \in \mathcal{E}_n$, as desired.
		\item Suppose that $\{A_m\}_{m\in\N}$ is a sequence of pairwise disjoint sets in $\mathcal{E}_n$.
		By $\sigma$-additivity of the measures $\leb_L$ and $\loeb$, then we have
		\begin{eqnarray*}
			\leb_L\left( \bigcup_{m \in \N} A_m \right) &=& \displaystyle \sum_{m\in\N} \leb_L(A_m) \text{ and}\\
			\loeb\left( \bigcup_{m \in \N} A_m \right) &=& \displaystyle \sum_{m\in\N} \loeb(A_m).
		\end{eqnarray*}
		Since we have assumed that $\leb_L(A_m) = \loeb(A_m)$ for all $m \in \N$, then also $\bigcup_{m \in \N} A_m \in \mathcal{E}$, as desired.
	\end{enumerate}
	
	We have verified that for every $n \in\N$ $\mathcal{C}_{n}$ is a $\pi$-system, $\mathcal{E}_n$ is a $\lambda$-system, and $\mathcal{C}_{n} \subseteq \mathcal{E}_{n}$ for every $n\in\N$.
	Then Dinkyn's $\pi$--$\lambda$ theorem ensures that the $\sigma$-algebra generated by $\mathcal{C}_{n}$ is a subset of $\mathcal{E}_n$ for every $n\in\N$.
	However, the $\sigma$-algebra generated by $\mathcal{C}_n$ is $\mathfrak{C}_n$, so that $\mathfrak{C}_n \subseteq \mathcal{E}_n$ for every $n\in\N$, as desired.
\end{proof}

By translation invariance of the measure $\loeb$, a similar result applies also to subsets of the translates $x+\hR_{fin}$, $x \in \hR$.
However, the above result cannot be extended over supersets of $\hR_{fin}$ (or to supersets of its translates $x+\hR_{fin}$, $x \in \hR$), since the Dinkyn's $\pi$--$\lambda$ theorem can only be applied to finite or $\sigma$-finite measurable sets.

The difference between the two measures can be explained in terms of the model-theoretic notions used in their definitions. In fact, the Loeb measure $\leb_L$ relies heavily on the properties of star transform, on the notion of internal sets and on the transfer principle of Robinson's framework. Instead, the uniform measure $\loeb$ is defined from first principles and does not exploit the strength of these notions. This difference explains the greater versatility of the Loeb measures and their applicability to a variety of mathematical problems. On the other hand, an advantage of the measure $\loeb$ is that it can be defined even for those field extensions of $\R$ where there is no analogous of a star transform, of a transfer principle or of a notion of internal sets.

\subsection{The real-valued measure in higher dimension and a real-valued integral}\label{sec Loeb integral}

In this section we generalize the definition of the real-valued measure $\loeb$ to $\F^n$ for all $n\in\N$.

\begin{definition}
	We say that a \emph{bounded rectangle} in $\F^n$ is the product of $n$ bounded intervals in $\F$.
	If $R=I_1, \ldots, I_n$ is a bounded rectangle, define $\loeb^n(R)=\sh{\left(l(I_1)\cdot\ldots\cdot l(I_n) \right)}$.
	
	For every $A \subseteq \F^n$ such that there exists a sequence of bounded rectangles $\{R_n\}_{n\in\N}$ satisfying $A \subseteq \bigcup_{n \in \N} R_n$,
	define
	$$
	\extl^n(A) = \inf\left\{\sum_{n \in \N}\loeb^n(R_n) : A \subseteq \bigcup_{n\in\N} R_n \right\}.
	$$
	If for every sequence of bounded rectangles $\{R_n\}_{n\in\N}$ we have $A \not\subseteq \bigcup_{n \in \N} R_n$, define $\extl^n(A) = +\infty$.
\end{definition}

As with the one-dimensional set function $\extl$, $\extl^n$ is an outer measure for all $n\in\N$.
Consequently, one can define the $\sigma$-algebra of measurable subsets of $\F^n$.
The definition is analogous to that of the Lebesgue integral in dimension $n$ from the Lebesgue measure in dimension $n+1$, as exposed for instance in \cite{pugh}.

\begin{definition}
	Given the outer measure $\extl$ on $\F^n$, the following family is
	called the \emph{Caratheodory $\sigma$-algebra} associated to $\extl$:
	$$
	\mathfrak{C}(\F^n)\ =\
	\left\{A\subseteq\F:
	\extl^n(B)=\extl^n(B\cap A)+\extl^n(B\setminus A)
	\text{ for all }B\subseteq\F^n\right\}.$$
	If $A \in \mathfrak{C}(\F_n)$, we will say that $A$ is $L$-measurable.
\end{definition}

The family $\mathfrak{C}(\F^n)$ is a $\sigma$-algebra, and that the restriction of $\extl^n$ to $\mathfrak{C}(\F^n)$, that we will denote by $\loeb^n$, is a complete measure.
As we have seen for the one-dimensional measure, the real-valued measures $\loeb^n$ are translation invariant and positively homogeneous. Moreover, by adapting the proof of Theorem \ref{thm coerenza L-meas e misura di lebesgue}, we obtain that if $A \subseteq \R^n$ is a Lebesgue measurable set, then $\st^{-1}(A) \subseteq \F^n$ is $L$-measurable and $\loeb^n(\st^{-1}(A))=\leb^n(A)$.

The $n$-dimensional measures can also be used to define a real-valued integral for functions over $\F$.

\begin{definition}
	Let $A\subseteq \F^n$, be a $L$-measurable set and let $f: A \rightarrow \F$ be a non-negative function.
	We say that $f$ is $L$-integrable iff
	$$
		\epi(f) = \{ (x_1, \ldots, x_n, x_{n+1})\in \F^{n+1} : 0 \leq x_{n+1} \leq f(x_1, \ldots, x_n) \}
	$$
	is $L$-measurable and $\loeb^{n+1}(\epi(f))<+\infty$.
	If $f: A \rightarrow \F$ is a non-negative $\loeb^n$-measurable function, we define
	$$
		\int_A f \ d\loeb^n = \loeb^{n+1}(\epi(f)).
	$$
	
	We say that $f: A \rightarrow \F$ is $L$-integrable iff $f^+$ and $f^-$ are.
	If $f: A \rightarrow \F$ is a $L$-integrable function, we define
	$$
	\int_A f \ d\loeb^n = \int_A f^+ \ d\loeb^n - \int_A f^- \ d\loeb^n.
	$$
\end{definition}

Thanks to additivity and positive homogeneity of the measure $\loeb^{n+1}$, the integral is $\R$-linear.
Moreover, the linearity property can be extended in the same spirit as positive homogeneity (see Proposition \ref{prop positive homogeneity}).

\begin{proposition}\label{linearity of the real-valued integral}
	If $A\subset \F^n$ is a $L$-measurable set,
	then for every $L$-integrable functions $f$ and $g$ over $A$ and for every $x,y \in \F_{fin}$,
	$$
	\int_A (xf+yg) \ d \loeb^n = \sh{x}\int_A f \ d\loeb^n + \sh{y}\int_A g \ d \loeb^n.
	$$
\end{proposition}
\begin{proof}
	Once we prove that the set $\epi(g)$ has the same measure as the set $\{(x_1, \ldots, x_n, x_{n+1})\in\F^{n+1}: f(x_1, \ldots, x_n) \leq x_{n+1} \leq g(x_1, \ldots, x_n) \}$, linearity is a consequence of the definition of the integral and of positive homogeneity of the measure $\loeb^{n+1}$.
	
	The proof that the two sets have the same measure can be obtained by adapting the proof of Theorem 16 \textit{(g)} of Chapter 6, Section 4 of \cite{pugh}.
	If $f$ and $g$ are both step functions, i.e. if both are defined over an interval and they are piecewise constant over subintervals of their domain, then the desired assertion is a consequence of translation invariance and positive homogeneity of the measure $\loeb^{n+1}$.
	If $f$ and $g$ are arbitrary, the proof relies on the possibility to approximate the sets $\epi(f)$, $\epi(g)$ and $\{(x_1, \ldots, x_n, x_{n+1})\in\F^{n+1}: f(x_1, \ldots, x_n) \leq x_{n+1} \leq g(x_1, \ldots, x_n) \}$ by rectangles up to an arbitrary precision.
\end{proof}

The integral of a function defined on a domain of a finite measure is invariant by infinitesimal perturbations of the function.

\begin{proposition}\label{proposition simeq implies same integral}
	Let $A \subseteq \F^n$ be a $L$-measurable set with $\loeb^n(A) < +\infty$.
	If $f: A \rightarrow \F$ is $L$-integrable and if $g: A \rightarrow \F$ satisfies $f(x)\simeq g(x)$ for every $x \in A$.
	Then
	$$
	\int_{A} f \ d\loeb^n = \int_A g \ d \loeb^n.
	$$
\end{proposition}
\begin{proof}
	By hypothesis over $f$ and $g$, $\epi(f-g) \subset A \times [-1/n,1/n]_{\F}$ for every $n\in\N$.
	As a consequence, $\loeb^{n+1}(\epi(f-g))\leq \frac{2}{n} \loeb^n(A)$. Since $\loeb^n(A) <+\infty$, $\epi(f-g)$ is a null set in $\F^{n+1}$.
	This and linearity of the integral is sufficient to entail the desired result.
\end{proof}

Notice that the above result relies in an essential way upon the hypothesis that $\loeb^n(A) <+\infty$.
A counterexample that does not satisfy this hypothesis is defined as follows. Let $\varepsilon \in\F$, $\varepsilon>0$ and $\varepsilon\sim0$.
Define $A = [0,\varepsilon^{-1}]_\F$, $f(x) = 0$ for all $x \in A$ and $g(x) = \varepsilon$ for all $x \in A$.
Then $\loeb(A) = +\infty$, $\int_A f \ d \loeb =0$ and $\int_A g \ d\loeb = 1$.

The integral of an integrable function whose range is $\F_{fin}$ is not affected also by infinitesimal perturbations of the domain.

\begin{proposition}\label{proposition values in fin implies same integral over infinitesimally close sets}
	Let $A \subseteq \F^n$ be a $L$-measurable set and let $B$ be a $L$-measurable set satisfying $\loeb(A \Delta B) =0$.
	If $f: A\cup B \rightarrow \F_{fin}$ is $L$-integrable, then
	$$
	\int_A f \ d \loeb^n = \int_B f \ d\loeb^n.
	$$
\end{proposition}
\begin{proof}
	Denote by $g$ the restriction of $f$ to $A\Delta B$.
	The hypotheses over $f$, $A$ and $B$ entail that $\epi(g) \subset (A \Delta B) \times [-\omega, \omega]_{\F}$ for every infinite $\omega \in \F$.
	This and the hypothesis that $\loeb(A\Delta B)=0$ entail $\int_{A \Delta B} f \ d\loeb^n =0.$
	The desired equality is a consequence of this result and of linearity of the integral.
\end{proof}

Similarly to Proposition \ref{proposition simeq implies same integral}, the hypotheses over $f$ are necessary.
A counterexample that does not satisfy this hypothesis is defined as follows. Let $\varepsilon \in\F$, $\varepsilon>0$ and $\varepsilon\sim0$.
Define $A = [0,1]_\F$, $f(x) = \varepsilon^{-1}$ for $x \in [0,\varepsilon]_\F$ and $f(x) =0$ otherwise.
If $B = [\varepsilon, 1]_\F$, then $A \Delta B = [0,\varepsilon]_{\F}$, so that $\loeb(A\Delta B) = 0$.
However, $\int_A f \ d \loeb = \int_{A \Delta B} f \ d \loeb = 1$ and $\int_{B} f \ d \loeb = 0$.

Finally, we can establish that every Lebesgue integrable function has a non-Archimedean representative that is $L$-integrable. Moreover, the integrals of the two functions assume the same value.

\begin{theorem}\label{lifting thm loeb integral}
	If $f: A \subseteq \R^n \rightarrow \R$ is Lebesgue integrable, then
	$\ext{0}{f} : \st^{-1}(A) \rightarrow \R$ defined by $\ext{0}{f}(x) = f(\sh{x})$ is $L$-integrable and $\int_{\st^{-1}(A)} \ext{0}{f} \ d\loeb = \int_{A} f \ dx$.
\end{theorem}
\begin{proof}
	Suppose at first that $\leb^{n}(A) < +\infty$.	
	Recall that $\int_A f \ dx = \leb^{n+1}(\epi(f))$ and that, by an argument analogous to that of Theorem \ref{thm coerenza L-meas e misura di lebesgue},  $\leb^{n+1}(\epi(f))=\loeb^{n+1}\left(\st^{-1}\epi(f)\right) <+\infty$.
	Thanks to the hypothesis that $\leb^n(A) < +\infty$, we can adapt the proofs of Propositions \ref{proposition simeq implies same integral} and \ref{proposition values in fin implies same integral over infinitesimally close sets} to obtain that $\loeb^{n+1}\left( \st^{-1}\epi(f)\right) = \loeb^{n+1}\left(\epi\left(\ext{0}{f}\right) \right)$, so that
	$$
	\int_A f \ dx = \leb^{n+1}(\epi(f))=
	\loeb^{n+1}\left(\st^{-1}\epi(f)\right) = \loeb^{n+1}\left(\epi\left(\ext{0}{f}\right)\right) = \int_{\st^{-1}(A)} \ext{0}{f} \ d \loeb^{n+1},
	$$
	as desired.
	
	If $\lambda^n(A) = +\infty$, we obtain the desired result by $\sigma$-additivity of the measures $\leb^n$ and $\loeb^n$ and by $\sigma$-finiteness of $\leb^n$.
\end{proof}

Thanks to Propositions \ref{proposition simeq implies same integral} and \ref{proposition values in fin implies same integral over infinitesimally close sets}, it is possible to sharpen Theorem \ref{lifting thm loeb integral}.

\begin{corollary}\label{lifting corollary loeb integral}
	If $A\subseteq \R^n $ is a Lebesgue measurable set with $\leb^n(A) <+\infty$, if $f: A \rightarrow \R$ is Lebesgue integrable and if $B \subseteq \F^n_{fin}$ is a $L$-measurable set that satisfies $\sh{B}=A$, then
	$\ext{0}{f} : B \rightarrow \F$ is $L$-integrable and $\int_{B} \ext{0}{f} \ d\loeb = \int_{A} f \ dx$.
	
	Moreover, if $g: B \rightarrow \F$ is $L$-integrable and $\sh{g(x)}=f(\sh{x})$ for every $x\in B$, then $\int_{B} g \ d\loeb = \int_{A} f \ dx$.
\end{corollary}

\section{Comparison with a class of non-Archimedean uniform measures}\label{section comparison with LC}

In this section we compare the real-valued measure $\loeb$ to a class of uniform measures that generalize the uniform measure developed by Shamseddine and Berz for the Levi-Civita field to arbitrary Cauchy complete non-Archimedean extensions of the real numbers.

\subsection{A class of non-Archimedean uniform measures}\label{section measure theory}


In analogy with the Lebesgue measure theory and following \cite{lpcivita,shamseddine2012,berz+shamseddine2003}, we say that a subset of $\F$ is $\m$-measurable if it can be approximated with arbitrary precision by a countable sequence of intervals.

\begin{definition}\label{measurable sets}
	A set $A \subseteq \F$ is $\m$-measurable if and only if for every $\varepsilon\in\F$ there exist two sequences of mutually disjoint intervals $\{I_n\}_{n\in\N}$ and $\{J_n\}_{n\in\N}$ such that
	\begin{enumerate}
		\item $\bigcup_{n\in\N} I_n \subseteq A \subseteq \bigcup_{n\in\N} J_n$;
		\item $\sum_{n\in\N} l(I_n)$ and $\sum_{n\in\N} l(J_n)$ converge in $\F$;
		\item $\sum_{n\in\N} l(J_n) - \sum_{n\in\N} l(I_n) \leq \varepsilon$.
	\end{enumerate}
\end{definition}

Notice that, according to the above definition, $\F$ is not $\m$-measurable, since if $\F \subseteq \bigcup_{n\in\N} J_n$, then $\sum_{n\in\N} l(J_n)$ does not converge in $\F$.
Moreover, the family of $\m$-measurable sets is not an algebra, since it is not closed under complements. In addition, due to the properties of convergence in non-Archimedean field extensions of $\R$, it is also not closed under countable unions. For further discussion on the family of measurable sets on the Levi-Civita field $\civita$, we refer to \cite{lpcivita,moreno,berz+shamseddine2003}.

Nevertheless, it is possible to define a $\F$-valued function, that we will still call a \emph{measure}, according to the convention established in \cite{berz+shamseddine2003}, on the family of $\m$-measurable sets under the additional hypothesis that $\F$ is Cauchy complete in the order topology.

\begin{lemma}
	Suppose that $\F$ is Cauchy complete in the order topology.
	Then for every $\m$-measurable set $A \subset \F$
	$$
	\underline{m}(A) = \sup \left\{ \sum_{n\in\N} l(I_n) : \{I_n\}_{n \in \N} \text{ is a sequence of mutually disjoint intervals with } \bigcup_{n \in \N} I_n \subseteq A \right\}
	$$
	and
	$$
	\overline{m}(A)\inf \left\{ \sum_{n\in\N} l(J_n) : \{J_n\}_{n \in \N} \text{ is a sequence of mutually disjoint intervals with } A \subseteq \bigcup_{n \in \N} J_n \right\}
	$$
	are well-defined.
	Moreover, $\underline{m}(A) = \overline{m}(A)$.	
\end{lemma}
\begin{proof}
	The proof can be obtained from the argument in Section 2 and Proposition 2.2 of \cite{berz+shamseddine2003}.
	Notice that this argument only depends upon the hypothesis that $\F$ is Cauchy complete in the order topology and does not rely on other properties of the Levi-Civita field.
\end{proof}

By exploiting the above Lemma, it is possible to define a measure for any $\m$-measurable set $A$.

\begin{definition}\label{definizione misura}
	Suppose that $\F$ is is Cauchy complete in the order topology.
	If $A\subset\F$ is a $\m$-measurable set, then the measure of $A$, denoted by $\m(A)$, is defined as
	$$
		\m(A) = \underline{m}(A) = \overline{m}(A).
	$$
\end{definition}

\begin{remark}
	If $F = \civita$, the Levi-Civita field, then the measure $\m$ of Definition \ref{definizione misura} is the uniform measure developed by Shamseddine and Berz.
	Its properties are discussed in detail in \cite{lpcivita,shamseddine2012,berz+shamseddine2003}.
\end{remark}

\subsection{The relation between $\loeb$ and the non-Archimedean uniform measures}

The main result of this section is that the real-valued measure $\loeb$ is compatible with the non-Archimedean unform measure $\m$ on every Cauchy complete non-Archimedean extension of the real numbers.
Namely, $\m$-measurable sets are also $L$-measurable, and the real-valued measure is equal to the standard part of the non-Archimedean one.

\begin{theorem}\label{thm coerenza L-meas e m-meas} 
	Suppose that $\F$ is Cauchy complete in the order topology.
	If $A \subset \F$ is $\m$-measurable, then $A \in \mathfrak{C}$ and $\sh{\m(A)} = \loeb(A)$.
\end{theorem}
\begin{proof}
	%
	Let $\{I_n\}_{n\in\N}$ and $\{J_n\}_{n\in\N}$ be two families of mutually disjoint intervals satisfying $\bigcup_{n\in\N} I_n \subseteq A \subseteq \bigcup_{n\in\N} J_n$ and $\sum_{n \in \N} l(J_n) - \sum_{n \in \N} l(I_n) \simeq 0$.
	As a consequence, $\sh{\m(A)} = \sh{\m\left(\bigcup_{n \in \N} I_n \right)} = \sh{\m\left(\bigcup_{n \in \N} J_n \right)}$.
	
	Notice that $\bigcup_{n\in\N} I_n$ and $\bigcup_{n\in\N} J_n$ are $L$-measurable, since they are a countable union of $L$-measurable sets.
	Since $\sum_{n\in\N} l(I_n)$ converges in $\F$, there exists $i \in \N$ such that $\sh(l(I_n))=l_L(I_n)=0$ for every $n > i$. 
	We deduce that
	$$
	\sh{\m\left(\bigcup_{n \in \N} I_n \right)}
	=
	\sh{\left( \sum_{n\in\N} l(I_n) \right)}
	=
	\sh{\left( \sum_{n\leq i} l(I_n) \right)} + \sh{\left( \sum_{n>i} l(I_n) \right)}
	=
	\sh{\left( \sum_{n\leq i} l(I_n) \right)}
	=
	\sum_{n \leq i} \sh{l(I_n)}.
	$$
	and
	$$
	\loeb\left(\bigcup_{n \in \N} I_n \right)
	=
	\sum_{n\in\N} l_L(I_n) = \sum_{n \leq i} l_L(I_n) = \sum_{n \leq i} \sh{l(I_n)}.
	$$
	From the previous equalities we conclude
	$$
	\loeb\left(\bigcup_{n \in \N} I_n \right)
	=
	\sh{\m\left(\bigcup_{n \in \N} I_n \right)}
	$$
	and, with a similar argument, we obtain also
	$$
	\loeb\left(\bigcup_{n \in \N} J_n \right)
	=
	\sh{\m\left(\bigcup_{n \in \N} J_n \right)}.
	$$
	
	Thus $\loeb\left(\bigcup_{n \in \N} I_n \right) =
	\sh{\m\left(\bigcup_{n \in \N} I_n \right)} = \sh{\m(A)} = \sh{\m\left(\bigcup_{n \in \N} J_n \right)} = \loeb\left(\bigcup_{n \in \N} J_n \right).$
	
	If $\loeb\left(\bigcup_{n\in\N} I_n\right) = \loeb\left(\bigcup_{n\in\N} J_n\right)$ is finite (and possibly equal to $0$), by Lemma \ref{lemma A c B c C}, we conclude that
	$A \in\mathfrak{C}$ and $\loeb(A) = \loeb\left(\bigcup_{n\in\N} I_n\right) = \loeb\left(\bigcup_{n\in\N} J_n\right) = \sh{\m(A)}$.
	
	If $\loeb\left(\bigcup_{n\in\N} I_n\right) = \loeb\left(\bigcup_{n\in\N} J_n\right) = +\infty$, we only need to prove that $A$ is $L$-measurable.
	Let $i \in \N$ such that $\sh(l(I_n))=l_L(I_n)=0$ for every $n > i$.
	Then $A = \left(\bigcup_{n=1}^i I_n \right) \cup N$, where $N \subset \civita$ is a $m$-measurable set with $\m(N) \simeq 0$.
	By the first part of the proof, $N$ is a $L$-null set, and in particular it is measurable.
	We have written $A$ as a finite union of intervals and of a $L$-null set.
	Since intervals are $L$-measurable and since $\mathfrak{C}$ is a $\sigma$-algebra, hence closed also for finite unions, we deduce that $A$ is also $L$-measurable, as desired.
\end{proof}

The previous result allows also to gauge the expressive power of the uniform measure $\m$.
In order to do so, we will exploit the standard part function $\sh $ in order to project subsets of $\F_{fin}$ to subsets of $\R$.
It turns out that $\m$-measurable subsets of $\F_{fin}$ 
are projected to the union of a set that is at most countable and of a finite union of intervals.

\begin{proposition}\label{proposition shadow of a measurable set}
	Suppose that $\F$ is Cauchy complete in the order topology.
	If $A\subset \F_{fin}$ is measurable, then there exists a finite (possibly empty) or countable set $C$, a natural number $n \in\N$ and $m$ intervals $K_1, \ldots, K_m$ such that
	$$\sh{A}=\{ \sh{x} : x\in A\} = C \cup \left(\bigcup_{j\leq m} K_j\right).$$
\end{proposition}
\begin{proof}
	Since $A \subseteq \F_{fin}$ is measurable, we deduce that $\m(A)$ is a finite number.
	Moreover, there are two families of mutually disjoint intervals $\{I_n\}_{n\in\N}$ and $\{J_n\}_{n\in\N}$ satisfying
	\begin{itemize}
		\item $\bigcup_{n\in\N} I_n \subseteq A \subseteq \bigcup_{n\in\N} J_n$,
		\item $\sum_{n \in \N} l(J_n) - \sum_{n \in \N} l(I_n) \simeq 0$ and
		\item $\m(A) - \sum_{n \in \N} l(I_n) \simeq 0$.
	\end{itemize}
	From the above properties and finiteness of $\m(A)$
	we deduce that $l(I_n)\in\civita_{fin}$ and that $l(J_n)\in\civita_{fin}$ for all $n \in\N$.
	
	Since each of the sums $\sum_{n \in \N} l(I_n)$ and $\sum_{n \in \N} l(J_n)$ converges in $\F$, let $\underline{i} = \sup\{n \in \N : l(I_n)\not\simeq 0\}$ and $\overline{i} = \sup\{n \in \N : l(J_n)\not\simeq 0\}$.
	Then for every $n \leq \underline{i}$ there exists $\underline{a}_n, \underline{b}_n \in \R$, $\underline{a}_n < \underline{b}_n$ such that $\sh{I_n} = [\underline{a}_n,\underline{b}_n]_\R$.
	On the other hand, if $n > \underline{i}$ then there exists $\overline{c}_n \in \R$ such that $\sh{I_n} = \{\overline{c}_n\}$.
	Similarly, for every $n \leq \overline{i}$ there exists $\overline{a}_n, \overline{b}_n \in \R$, $\overline{a}_n < \overline{b}_n$ such that $\sh{J_n} = [\overline{a}_n,\overline{b}_n]_\R$ and for every $n > \overline{i}$ then there exists $\overline{c}_n \in \R$ such that $\sh{I_n} = \{\overline{c}_n\}$.
	We obtain the inclusions
	$$
	\left\{\underline{c}_n : n > \underline{i}\right\}
	\cup
	\bigcup_{n \leq \underline{i}} [\underline{a}_n,\underline{b}_n]_\R
	\subseteq
	\sh{A}
	\subseteq
	\{\overline{c}_n : n > \overline{i}\}
	\cup
	\bigcup_{n \leq \overline{i}} [\overline{a}_n,\overline{b}_n]_\R.
	$$
	
	If we prove that $\bigcup_{n \leq \underline{i}} [\underline{a}_n,\underline{b}_n]_\R = \bigcup_{n \leq \overline{i}} [\overline{a}_n,\overline{b}_n]_\R$, then we obtain the desired result.
	
	Notice that
	$$\sum_{n \in \N} l(J_n) - \sum_{n \in \N} l(I_n)
	=
	\left(\sum_{n \leq \overline{i}} l(J_n)-\sum_{n \leq \underline{i}} l(I_n)\right) + \left(\sum_{n > \overline{i}} l(J_n) - \sum_{n > \underline{i}} l(I_n)\right).
	$$
	Since $l(J_n)\simeq 0$ whenever $n > \overline{i}$ and $l(I_n)\simeq0$ whenever $n > \underline{i}$, Lemma 2.11 of \cite{lpcivita} implies that $\sum_{n > \overline{i}} l(J_n) - \sum_{n > \underline{i}} l(I_n) \simeq 0$.
	This result and the hypothesis that $\sum_{n \in \N} l(J_n) - \sum_{n \in \N} l(I_n) \simeq 0$ entail the equality
	$$
	\sum_{n \leq \overline{i}} l(J_n)=\sum_{n \leq \underline{i}} l(I_n).
	$$
	This can only happen if $\bigcup_{n \leq \underline{i}} [\underline{a}_n,\underline{b}_n]_\R = \bigcup_{n \leq \overline{i}} [\overline{a}_n,\overline{b}_n]_\R$, as desired.
	
	As a consequence, we can choose e.g. $m = \underline{i}$, $K_j = [\underline{a}_j,\underline{b}_j]$ for every $j\leq m$ and $C = \bigcup_{n \geq \overline{i}} \{\overline{c}_n\} \cap A$.
\end{proof}

The above result and Proposition \ref{proposition shadow of a measurable set} entail that the real-valued measure $\loeb$ allows for more measurable sets than the non-Archimedean measure $\m$.

\section{A real-valued integral on the Levi-Civita field}\label{sec integral}

In this section we exploit the same ideas used in the construction of the real-valued measure
and introduce a real-valued integral on functions defined on the Levi-Civita field.
This real-valued integral is obtained from the non-Archimedean integral defined by Shamseddine and Berz in \cite{berz+shamseddine2003} and further developed in \cite{lpcivita,shamseddine2012}.
Recently, the integration on the Levi-Civita field has also been extended in dimension $2$ and $3$ by Flynn and Shamseddine \cite{shamflin2,shamflin1}, but in our paper we focus only on the integration in dimension $1$.

We start by recalling the basic definitions and properties of the Levi-Civita field and of its non-Archimedean integral.

\subsection{The Levi-Civita field}\label{sec introcivita}


\begin{definition}
	A set $F \subset \Q$ is called left-finite if and only if for every $q \in \Q$ the set $\{x \in F : x \leq q \}$ is finite.
	The Levi-Civita field is the set $$\civita = \{ x:\Q \rightarrow \R: \{q:x(q)\not=0\} \text{ is left-finite} \},$$
	with the pointwise sum and the product defined by the formula
	$$
	(x \cdot y)(q) = \sum_{q_1+q_2=q} x(q_1) \cdot y(q_2).
	$$
\end{definition}

For a review of the algebraic and topological properties of $\civita$, we refer for instance to \cite{berz, analysislcf, reexp, Shamseddinephd} and references therein.

In the Levi-Civita field there are two notions of convergence: the one induced by the metric, analogous to the usual definition of limit for real-valued sequences, usually called strong convergence, and the weak convergence.

\begin{definition}\label{definition strong convergence}
	A sequence $\{a_n\}_{n \in \N}$ of elements of $\civita$ strongly converges to $l \in \civita$ if and only if
	$$
	\forall \varepsilon \in \civita, \varepsilon > 0, \exists n \in \N : \forall m > n\ |c_m-l|<\varepsilon.
	$$
	
	%
	A sequence $\{a_n\}_{n \in \N}$ of elements of $\civita$ weakly converges to $l \in \civita$ if and only if
	$$
	\forall \varepsilon \in \R, \varepsilon > 0, \exists n \in \N : \forall m > n\ \max_{q \in \Q, q\leq \varepsilon^{-1}} \left| (c_m - l)(q) \right|<\varepsilon.
	$$
	We will denote weak convergence with the expression $\wlim_{n \rightarrow \infty} a_n = l$.
\end{definition}



If $a_n \in \civita$ for all $n\in\N$ and if $x_0\in\civita$, we assume that the expression $\sum_{n \in \N} a_n (x-x_0)^n$ denotes the weak limit $\wlim_{k \rightarrow \infty} \sum_{n \leq k} a_n (x-x_0)^n$.



\begin{definition}
	We denote by $\P(I(a,b))$ the algebra of all power series that weakly converge for every $x \in I(a,b)$.
\end{definition}

Since power series with real coefficients weakly converge also in $\civita$ (we refer to \cite{lpcivita} for more details), it is possible to use them to define several extensions of real continuous functions to functions defined on the Levi-Civita field.
These extensions are obtained from the Taylor series expansion of a function at a point.

\begin{definition}\label{def ext}
Let $f \in C^\infty([a,b])$.
The analytic extension of $f$ is defined as
$$
\ext{\infty}{f}(r+\varepsilon) = \sum_{i = 0}^{\infty} f^i(r) \frac{\varepsilon^i}{i!}.
$$
for all $r \in [a,b]$ and for all $\varepsilon \in M_o$ such that $r+\varepsilon \in [a,b]_{\civita}$ (i.e.\ on the nearstandard points of $[a,b]_\civita$).
If $f$ is analytic, then $\ext{\infty}{f}$ will be called the canonical extension of $f$.
The only exceptions are the canonical extensions of the exponential function and of the trigonometric functions sine and cosine, still denoted by $e^x$, $\sin(x)$ and $\cos(x)$ for all $x \in \civita$ with $\lambda(x)\geq0$.

Let $f \in C^n([a,b])$, possibly with $n = \infty$.
The order $k$ extension of $f$, with $0\leq k\leq n$, is defined by
$$
\ext{k}{f}(r+\varepsilon) = \sum_{i = 0}^{k} f^i(r) \frac{\varepsilon^i}{i!}.
$$
for all $r \in [a,b]$ and for all $\varepsilon \in M_o$ such that $r+\varepsilon \in [a,b]_{\civita}$.
\end{definition}

It is possible to prove that each of the extensions introduced in Definition \ref{def ext} is unique and well-defined.
The functions $\ext{k}{f}$ extend the corresponding real function $f\in C^n([a,b])$ in the sense that for every $x \in [a,b]$ $f(x)=\ext{k}{f}$ for all $k \leq n$, possibly with $n = \infty$.
For more properties of the continuations of real functions to the Levi-Civita field, we refer to \cite{berz,analysislcf,reexp,bottazzi,lpcivita,Shamseddinephd}.

\subsection{Integration on the Levi-Civita field}

We briefly recall the basic notions of the non-Archimedean integration on the Levi-Civita field, following the approach of \cite{lpcivita}, that differs from the one of \cite{shamseddine2012,berz+shamseddine2003} in the definition of the simple and measurable functions.

As with the Lebesgue measure, the family of measurable functions is obtained from a family of simple functions.
For the remainder of the paper, we will work with the family of simple functions $\P=\bigcup_{a,b \in \civita, a<b}\P(I(a,b))$: as a consequence, a function $f$ is simple iff there exists an interval $I$ such that $\supp f= I$ and $f$ is a power series that converges for every $x\in I$.

\begin{proposition}\label{prop simple}
	If $f$ is simple on $I(a,b)$, then there exists a unique simple function $g:[a,b]_\civita\rightarrow \civita$ such that $g_{|I(a,b)}=f$.
\end{proposition}
\begin{proof}
	See \cite{lpcivita}.
\end{proof}

With a slight abuse of notation, if $f$ is simple on $I(a,b)$, we will still denote by $f$ the simple function defined on $[a,b]_\civita$ that coincides with $f$ on $I(a,b)$.

From the algebra of simple functions it is possible to define the family of measurable functions.
Following \cite{lpcivita}, we do not require that measurable functions must be bounded.

\begin{definition}\label{measurable functions}
	Let $A\subset \civita$ be measurable and let $f: A \rightarrow \civita$. 
	The function $f$ is measurable iff for all $\varepsilon\in \civita$, $\varepsilon > 0$, there exists a sequence of mutually disjoint intervals $\{I_n\}_{n\in\N}$ such that
	\begin{enumerate}
		\item $\bigcup_{n\in\N} I_n \subset A$;
		\item $\sum_{n\in\N} l(I_n)$ strongly converges in $\civita$;
		\item $\m(A)-\sum_{n\in\N} l(I_n)<\varepsilon$;
		\item for all $n \in \N$, $f$ is simple on $I_n$.
	\end{enumerate}
	
	We will denote by $\meas(A)$ the set of measurable functions on $A$.
\end{definition}

Since every simple function has an antiderivative, it is possible to define the integral of a simple function over an interval by imposing the validity of the fundamental theorem of calculus.
The integral of a measurable function over a measurable set can then be obtained as a limit of the integrals of simple functions over a sequence of intervals satisfying Definition \ref{measurable functions}.

\begin{definition}\label{def integral}
	If $f$ is a simple function over $I(a,b)$ whose antiderivative is $F$, then
	$$
	\int_{I(a,b)} f(x)  = \lim_{x\rightarrow b} F(x) - \lim_{x \rightarrow a} F(x).
	$$
	Notice that the two limits in the previous equality are well-defined, since $F$ is simple on $I(a,b)$ and, thanks to Proposition \ref{prop simple}, $F$ can be extended to a simple function on $[a,b]_\civita$.
	
	If $A \subset \civita$ is a measurable set and $f: A \rightarrow \civita$ is a measurable function, then define
	$$
	F_A= \left\{ \{I_n\}_{n\in\N} : \bigcup_{n \in \N} I_n \subseteq A,\ I_n \text{ are mutually disjoint and } \forall n \in \N\ f \text{ is simple on } I_n \right\}.
	$$
	The integral of $f$ over $A$ is defined as
	$$
	\int_{A} f(x)  = \lim_{\{I_n\}_{n\in\N}\in F_A,\ \sum_{n \in \N} I_n \rightarrow \m(A)} \left( \sum_{n \in \N}  \int_{I_n} f(x) \right)
	$$
	whenever the limit on the right side of the equality is defined (and possibly equal to $\pm \infty$ whenever the sequence $k \mapsto \sum_{n \leq k}  \int_{I_n} f(x)$ diverges), and it is undefined otherwise.
\end{definition}

The integral on the Levi-Civita field is coherent with the Lebesgue integral.

\begin{lemma}\label{lemma coherence}
	If $a, b \in \R$ and $f \in \an([a,b])$, then $\ext{\infty}{f}$ is measurable and
	$$\int_{[a,b]} \ext{\infty}{f}  = \int_{[a,b]} f(x) dx .$$
	Moreover, if $I(c,d) \subseteq [a,b]_\civita$,  then
	$$\int_{I(c,d)}\ext{\infty}{f}  \approx \int_{[c[0],d[0]]} f(x) dx .$$
\end{lemma}

However, notice that extensions of the form $\ext{k}{f}$, with $k \in \N$, are not measurable over sets of a non-infinitesimal measures, since they are not simple over such sets.

In analogy with the real measure theory, it is possible to introduce spaces of measurable functions whose $p$-th power has a well-defined integral.

\begin{definition}
	Let $A \subset \civita$ be a measurable set.
	If $1\leq p < \infty$, define
	$$\L^p(A)=\left\{ [f] : f \text{ and } f^p \text{ are measurable}, \int_A |f|^p \text{ is defined and } \int_A |f|^p  < + \infty \right\}.$$
	If $f \in \L^p(A)$, then define $\norm{f}_p = \left(\int_A |f|^p \right)^{1/p}$.
\end{definition}

The properties of the $\L^p$ spaces are studied in detail in \cite{lpcivita}.

\subsection{A real-valued integral on the Levi-Civita field}

We now introduce a real-valued integral on the Levi-Civita field by exploiting a similar idea as the one used in Section \ref{sec loeb} for the introduction of the real-valued measure $\loeb$.
In order to avoid confusion with the $L$-integral introduced in Section \ref{sec Loeb integral}, we will denote this new integral as \civintegral. The letter $M$ suggests dependence of the integral upon the family of measurable functions over the Levi-Civita field introduced in Definition \ref{measurable functions}.

\begin{definition}\label{def L integrable for bounded functions}
	Let $A \subset \civita$ be a measurable set and let $f: A \rightarrow \civita$ be a bounded 
	function. 
	$f$ is \civint\ iff
	\begin{eqnarray*}
		&\displaystyle
		\sup\left\{\sh{\left(\int_A g \right)} : g\in\L^1(A) \text{ and } g(x)\leq f(x) \ \forall x \in A\right\}\\
		&=\\
		&\displaystyle
		\inf\left\{\sh{\left(\int_A g \right)} : g\in\L^1(A) \text{ and } g(x)\geq f(x) \ \forall x \in A\right\}.
	\end{eqnarray*}
	If $f$ is bounded and \civint, we define
	\begin{eqnarray*}
		\displaystyle
		\int_A^M f
		&=&
		\sup\left\{\sh{\left(\int_A g \right)} : g\in\L^1(A) \text{ and } g(x)\leq f(x) \ \forall x \in A\right\}\\
		&=&
		\inf\left\{\sh{\left(\int_A g \right)} : g\in\L^1(A) \text{ and } g(x)\geq f(x) \ \forall x \in A\right\}.
		\displaystyle
	\end{eqnarray*}
\end{definition}

Through this section, we will  prove that measurable functions in $\L^1$ are also $M$-integrable, and the standard part of the non-Archimedean integral is equal to the real-valued integral.
Currently, we can only prove this result for bounded measurable functions.

\begin{proposition}\label{prop L^1 bounded implica Loeb 1}
	Let $A \subset \civita$ be a measurable set.
	If $f \in \L^1(A)$ is bounded, then $f$ is \civint\ over $A$.
	Moreover, $\sh{\left(\int_{A} f\right)} = \int_{A}^M f$.
\end{proposition}
\begin{proof}
	If $f \in \L^1(A)$, then
	$$
	f \in
	\left\{g\in\L^1(A) : g(x)\leq f(x) \ \forall x \in A\right\}
	\cap
	\left\{g\in\L^1(A) : g(x)\geq f(x) \ \forall x \in A\right\}.
	$$
	Thus $\sh{\left(\int_{A} f\right)} \leq \int_{A}^M f \leq \sh{\left(\int_{A} f\right)}$, i.e. $\sh{\left(\int_{A} f\right)} = \int_{A}^M f$.
\end{proof}

\begin{corollary}\label{prop L^p bounded implica Loeb p}
	Let $A \subset \civita$ be a measurable set.
	Then, for all $1\leq p <+\infty$, if $f \in \L^p(A)$ is bounded, $f^p$ is \civint\ and $\int_A^M f^p = \sh\left(\int_A^M f^p\right)$.
\end{corollary}
\begin{proof}
	If $f \in \L^p(A)$ is bounded, then $|f|^p \in \L^1(A)$ is also bounded. 
	By Corollary 4.12 of \cite{berz+shamseddine2003}, $\left| \int_A f^p \right| \leq \int_A \left| f^p \right|$. Since $f \in \L^p(A)$, $\int_A \left| f^p \right|$ is defined and it is not equal to $+\infty$.
	Thus $f^p \in \L^1(A)$. By Proposition \ref{prop L^1 bounded implica Loeb 1} we conclude that $\int_A^M f^p = \sh\left(\int_A^M f^p\right)$, as desired.
\end{proof}

The definition of \civint\ functions can be extended to unbounded functions in the usual way.

\begin{definition}
	Let $A \subset \civita$ be a measurable set and let $f: A \rightarrow \civita$ be an unbounded, non-negative function. 
	We say that $f$ is \civint\ iff
	$$
	\sup\left\{\sh{\left( \int_A g \right)}  : g\in\L^1(A) \text{ and } g(x)\leq f(x) \ \forall x \in A\right\}<+\infty
	$$
	If $f$ is unbounded, non-negative and \civint, we define
	$$
	\int_A^M f
	=
	\sup\left\{\sh{\left( \int_A g \right)} : g\in\L^1(A) \text{ and } g(x)\leq f(x) \ \forall x \in A\right\}.
	$$
	
	If $f$ is unbounded and non-positive, we say that $f$ is \civint\ iff $f^-$ is, and we define $\int_A^M f = -\int_A^M f^-$.
	
	If $f$ is unbounded, we say that $f$ is \civint\ iff $f^+$ and $f^-$ are \civint.
	If $f$ is unbounded and \civint, we define $\int_A^M f = \int_A^M f^+ -\int_A^M f^- $.
\end{definition}

The real-valud integral is linear.

\begin{proposition}\label{prop linearita l integral}
	If $A\subset \civita$ is a measurable set,
	then for every \civint\ functions $f$ and $g$ over $A$ and for every $x,y \in \civita_{fin}$,
	\begin{equation}\label{eqn linearita}
		\int_A^M (xf+yg) = \sh{x}\int_A^M f + \sh{y}\int_A^M g.
	\end{equation}
\end{proposition}
\begin{proof}
	If $f$ and $g$ are non-negative and if $x$ and $y$ are non-negative, then linearity of the non-Archimedean integral over the Levi-Civita field and the properties of the standard part entail the equalities
	\begin{eqnarray*}
		&\displaystyle
		\sup\left\{\sh{\left(\int_A h \right)} : h\in\L^1(A) \text{ and } h(x)\leq xf(x)+yg(x) \ \forall x \in A\right\}\\
		&=\\
		&\displaystyle
		\sh{x}\sup\left\{\sh{\left(\int_A h \right)} : h\in\L^1(A) \text{ and } h(x)\geq f(x) \ \forall x \in A\right\}\\
		&+\\
		&\displaystyle
		\sh{y}\sup\left\{\sh{\left(\int_A h \right)} : h\in\L^1(A) \text{ and } h(x)\geq g(x) \ \forall x \in A\right\}.
	\end{eqnarray*}
	This is sufficient to deduce that equation \eqref{eqn linearita} is satisfied.
	
	For arbitrary $f, g, x, y$, the desired property can be obtained by applying the previous part of the proof to $(xf)^+$, $(xf)^-$, $(yg)^+$ and $(yg)^-$.
\end{proof}

Finally, we are ready prove that all $\m$-measurable functions in $\L^1$, even those that are not bouunded, are \civint, and that their \civintegral\ is equal to the standard part of the integral of Definition \ref{def integral}.

\begin{proposition}\label{prop L^1 implica Loeb 1}
	Let $A \subset \civita$ be a measurable set.
	If $f \in \L^1(A)$ and if $\int_{A} |f| \in\civita_{fin}$, then $f$ is \civint\ and
	$$
	\int_A^M f = \sh{\left(\int_A f \right)}.
	$$
\end{proposition}
\begin{proof}
	If $f \in \L^1(A)$ is bounded, then the desired result is entailed by Proposition \ref{prop L^1 bounded implica Loeb 1}.
	
	Suppose then that $f$ is unbounded and non-negative, and that $\lambda\left(\int_{A} |f|\right)\geq 0$.
	As a consequence of the latter hypothesis we have also $\sh{\left( \int_A f \right)}<+\infty$, so that $f$ is \civint.
	Moreover,
	$$
	f \in
	\left\{g\in\L^1(A) : g(x)\leq f(x) \ \forall x \in A\right\}
	$$
	so that $\sh{\left(\int_{A} f\right)} \leq \int_{A}^M f$.
	
	Let now $\varepsilon\in \civita$, $\varepsilon > 0$ satisfying $0<\lambda(\varepsilon \cdot \m(A))<+\infty$.
	Define also a function $\tilde{f}: A \rightarrow \civita$ by posing $\tilde{f}(x)=f(x)+\varepsilon$.
	Then $\tilde{f}\not\in \left\{g\in\L^1(A) : g(x)\leq f(x) \ \forall x \in A\right\}$.
	As a consequence,
	$$
	\int_A^M f \leq \sh{\left(\int_A \tilde{f} \right)} 
	=
	\sh{\left(\int_A f + \int_A \varepsilon \right)}
	=
	\sh{\left( \int_A f \right)} + \sh{(\varepsilon \cdot \m(A))}.
	$$
	Since $0<\lambda(\varepsilon \cdot \m(A))<+\infty$, $\sh{(\varepsilon \cdot \m(A))}=0$ and $\int_A^M f \leq \sh{\left( \int_A f \right)}$.
	
	We conclude  $\sh{\left(\int_{A} f\right)} \leq \int_{A}^M f \leq \sh{\left( \int_A f \right)}$, as desired.
	
	If $f$ is unbounded, non-positive and with $\lambda\left(\int_{A} |f|\right)\geq 0$, we can apply the above argument to $-f$.
	
	For an arbitrary $f\in L^1(A)$, 
	the previous arguments entail that $\int_A^M f^\pm =  \sh{\left(\int_{A} f^\pm\right)}$, so that
	$$
	\int_A^M f = \int_A^M f^+ - \int_A^M f^-
	=
	\sh{\left(\int_{A} f^+\right)}-\sh{\left(\int_{A} f^-\right)}
	=
	\sh{\left(\int_{A} f^+ - \int_{A} f^-\right)}
	=
	\sh{\left(\int_{A} f\right)},
	$$
	as desired.
\end{proof}

As a consequence of the coherence between the real-valued integral and the integral in $\civita$, we obtain the following result, analogous to Corollary \ref{prop L^p bounded implica Loeb p}.

\begin{corollary}\label{prop L^p implica Loeb p}
	Let $A \subset \civita$ be a measurable set.
	If $f \in \L^p(A)$, then $f^p$ is \civint\ and $\int_A^M f^p = \sh\left(\int_A f^p\right)$.
\end{corollary}

\subsection{Integrals over nonmeasurable sets}

It is also possible to define a real-valued integral of functions defined upon some non-measurable sets $A \subseteq \civita$.
For a relevant class of non-measurable sets, the definition is analogous to that of the Riemann integral over unbounded sets.

\begin{definition}\label{def integral over nonmeasurable sets}
	For every $q \in\Q$, define $A(q) = \{ x \in \civita : \lambda(x) \geq q \}$ and $B(q) = \{x \in \civita : \lambda(x)>q\}$.
	
	We say that a function $f: A(q) \rightarrow \civita$ is \civint\ if
	$$
	\lim_{t\rightarrow\infty} \int_{[-td^q,td^q]}^M f
	$$
	exists finite in $\R$.
	If $f$ is \civint\ over $A(q)$, we define $\int_{A(q)}^M f(x) = \lim_{t\rightarrow\infty} \int_{[-td^q,td^q]}^M f(x)$.
	
	Similarly, we say that a function $f: B(q) \rightarrow \civita$ is \civint\ if the real limit
	$$
	\lim_{t\rightarrow q^+} \int_{[-d^t,d^t]}^M f
	$$
	exists finite in $\R$.
	If $f$ is \civint\ over $B(q)$, we define $\int_{B(q)}^M f(x) = \lim_{t\rightarrow q^+} \int_{[-d^t,d^t]}^M f(x)$.
	
	Finally, we say that a function $f : \civita \rightarrow \civita$ is \civint\ if the real limit
	$$
	\lim_{t\rightarrow -\infty} \int_{[-d^t,d^t]}^M f
	$$
	exists finite in $\R$.
	If $f$ is \civint\ over $\civita$, we define $\int_{\civita}^M f(x) = \lim_{t\rightarrow -\infty} \int_{[-d^t,d^t]}^M f(x)$.
	
	These definitions are extended in the expected way to sets of the form $A(q) \cap [a,b]_\civita$, $B(q)\cap [a,b]_\civita$, $[a,+\infty]_\civita$ and $[-\infty,b]_\civita$.
\end{definition}

These integrals over non-measurable subsets of $\civita$ are linear, as a consequence of linearity of the integral over measurable sets.

\begin{proposition}\label{prop linearita l integral 2}
	If $A\subset \civita$ is a measurable set or if it is a set of the form $A(q)$ or $B(q)$ for some $q\in\Q$,
	then every \civint\ functions $f$ and $g$ over $A$ and every $x,y \in \civita_{fin}$ satisfy equality \eqref{eqn linearita}.
\end{proposition}
\begin{proof}
	We have proven this result under the hypothesis that $A$ is a measurable set in Proposition \ref{prop linearita l integral}.
	
	If $A = A(q)$ or $A=B(q)$ for some $q\in\Q$, then this is a consequence of the linearity of the integral over measurable sets and of the linearity of the real limit of functions.
\end{proof}

\begin{example}
	Let $a \in \R$, $a\ne -1$, and consider the function $f(x)=x^{a}$ defined on the set $\{x \in \civita_{fin} : x \geq 1\} = A(0) \cap [1,d^{-1}]_\civita$.
	These functions are analytic over intervals of the form $[1,n]_{\civita}$, with $n \in \N$.
	However, recall that they are not integrable over the non-measurable set $\{x \in \civita_{fin} : x \geq 1\}$.
	Nevertheless, Proposition \ref{prop L^1 bounded implica Loeb 1} entails the equality $\int_{[1,n]_{\civita}}^M x^a = \int_{[1,n]_{\civita}} x^a$ for every $n \in\N$.
	Moreover, by Lemma 4.10 of \cite{lpcivita}, $\int_{[1,n]_{\civita}} x^a = \int_{[1,n]} x^a \ dx = \frac{n^{a+1}}{a+1}-\frac{1}{a+1}$.
	
	A similar argument applies to the function $x \mapsto x^{-1}$. In this case, we obtain that $\int_{[1,n]_{\civita}}^M x^{-1} = \int_{[1,n]_{\civita}} x^{-1} = \int_{[1,n]} x^{-1} \ dx = \ln(n)$ for every $n \in\N$.
	
	As a consequence, $f(x)=x^{a}$ is \civint\ over $\{x \in \civita_{fin} : x \geq 1\}$ if and only if $a<-1$, in analogy to what happens for the corresponding real functions.
\end{example}

\subsection{\civint\ representatives of real continuous functions}
We will now establish an analogous of Theorem \ref{lifting thm loeb integral} for the \civintegral.
Recall that, for the non-Archimedean integral, the only analogue to Theorem \ref{lifting thm loeb integral} is Lemma \ref{lemma coherence}, that is only valid for real analytic functions of a bounded domain.
Instead, for the \civintegral\ we will prove that every canonical or non-canonical extension of a continuous function $f \in C^k([a,b])$ is \civint, and the \civintegral\ is equal to the Lebesgue integral of $f$ over $[a,b]$.

\begin{theorem}\label{lifting thm}
	Let $f \in C^k([a,b])$. 
	Then $\ext{h}{f}$ is \civint\ for every $0 \leq h \leq k$ and
	$$
	\int_{[a,b]_\civita}^M \ext{h}{f} = \int_{[a,b]} f(x) dx.
	$$
\end{theorem}
\begin{proof}
	Recall that a function $s: [a,b] \rightarrow \R$ is a step function if and only if there exist some $n \in\N$, a partition of $[a,b]$ into $n$ intervals $I_1, \ldots, I_n$ and $n$ real numbers $s_1, \ldots, s_n$, such that $s = \sum_{i\leq n} s_i \chi_{I_i}$.
	Step functions have a well-defined Riemann integral equal to $\sum_{i\leq n} s_i l(I_i)$.
	
	If $I_1, \ldots, I_n$ is a partition of $[a,b]$, define $J_1, \ldots, J_n$ as the partition of $[a,b]_{\civita}$ satisfying the following properties:
	\begin{itemize}
		\item if $I_i = [a_i,b_i]$, then $J_i = [a_i,b_i]_{\civita}$;
		\item if $I_i = (a_i,b_i]$, then $J_i = (a_i,b_i]_{\civita}$;
		\item if $I_i = [a_i,b_i)$, then $J_i = [a_i,b_i)_{\civita}$;
		\item if $I_i = (a_i,b_i)$, then $J_i = (a_i,b_i)_{\civita}$;
	\end{itemize}
	
	If $s: [a,b] \rightarrow \R$ is a step function, then it can be extended to a function $\tilde{s}:[a,b]_{\civita} \rightarrow \civita$ by posing
	$$
		\tilde{s} = \sum_{i\leq n} s_i J_i.
	$$
	The function $\tilde{s}$ is trivially $\m$-measurable, moreover
	\begin{equation}\label{eqn x lifting}
		\int_{[a,b]_\civita} \tilde{s} = \int_{[a,b]_\civita}^M \tilde{s} = \int_{[a,b]} s(x) \ dx.
	\end{equation}

	Let now $f\in C^k([a,b])$ and let $0 \leq h \leq k$. For every $\varepsilon \in\R, \varepsilon > 0$ let also $s_\varepsilon^+$ be a step function satifying $s_\varepsilon^+(x)> f(x)$ for every $x \in [a,b]$ and 
	\begin{equation}\label{eqn eccesso x lifting}
		\int_{[a,b]} s_\varepsilon^+(x) dx - \int_{[a,b]} f(x) \leq \varepsilon.
	\end{equation}
	Similarly, let $s_\varepsilon^-$ be a step function satifying $s_\varepsilon^-(x)< f(x)$ for every $x \in [a,b]$ and 
	\begin{equation}\label{eqn difetto x lifting}
		\int_{[a,b]} f(x) dx - \int_{[a,b]} s_\varepsilon^-(x) \leq \varepsilon.
	\end{equation}
	
	Notice that, by definition, we have $\tilde{s}_\varepsilon^+(x) > \ext{h}{f}(x)$ for every $x \in [a,b]_\civita$ and $\tilde{s}_\varepsilon^-(x) < \ext{h}{f}(x)$ for every $x \in [a,b]_\civita$.
	Moreover, equations \eqref{eqn x lifting}, \eqref{eqn eccesso x lifting} and \eqref{eqn difetto x lifting} entail that $\int_{[a,b]_\civita} \tilde{s}_\varepsilon^+ - \int_{[a,b]_\civita} \tilde{s}_\varepsilon^- < 2\varepsilon$.
	
	Since $\varepsilon$ is an arbitrary positive real number, we deduce that
	\begin{eqnarray*}
		 \displaystyle
		\sup\left\{\sh{\left(\int_{[a,b]_\civita} g \right)} : g(x)\leq \ext{h}{f}(x) \ \forall x \in [a,b]_\civita\right\} & =\\
		 \displaystyle
		\inf\left\{\sh{\left(\int_{[a,b]_\civita} g \right)} :g(x)\geq \ext{h}{f}(x) \ \forall x \in [a,b]_\civita\right\}
		&=&
		\int_{[a,b]} f(x) dx\\
	\end{eqnarray*}
	i.e.\ that $\ext{h}{f}$ is \civint\ and $\int_{[a,b]_\civita}^M \ext{h}{f} = \int_{[a,b]} f(x) dx$.
\end{proof}

A similar result applies also to continuous and integrable functions defined over open intervals or over $\R$.

\begin{corollary}\label{lifting corollary 1}
	If $f \in C^k((a,b))$ and $\int_{(a,b)} |f(x)| dx < +\infty$, then $\ext{h}{f}$ is \civint\ for every $0 \leq h \leq k$ and
	$$
	\int_{(a,b)}^M \ext{h}{f} = \int_{(a,b)} f(x) dx.
	$$
\end{corollary}
\begin{proof}
	If $f$ is bounded, it is sufficient to apply the same argument of Theorem \ref{lifting thm}.
	
	If $f$ is not bounded and non-negative, it is sufficient to notice that
	$$
	\sup\left\{\int_A s(x) \ dx  : s \text{ is a step function and } s(x)< f(x) \ \forall x \in (a,b)\right\}
	= \int_{(a,b)} f(x) \ dx
	$$
	and that, with a similar argument as the one used in the proof of Theorem \ref{lifting thm},
	\begin{eqnarray*}
		& \displaystyle 
		\sup\left\{\sh{\left(\int_{(a,b)_\civita} g \right)} : g(x)\leq \ext{h}{f}(x) \ \forall x \in (a,b)_\civita\right\}\\
		& \displaystyle =\\
		& \displaystyle 
		\sup\left\{\int_{(a,b)_\civita} s(x) \ dx  : s \text{ is a step function and } s(x)< f(x) \ \forall x \in (a,b)\right\}.
	\end{eqnarray*}
	For arbitrary functions $f \in C^k((a,b))$, it is sufficient to apply the previous part of the proof to $f^+$ and to $f^-$.
	Notice also that the hypothesis $\int_{(a,b)} |f(x)| dx < +\infty$ is sufficient to entail that both $\int_{(a,b)} f^+(x) dx < +\infty$ and $\int_{(a,b)} f^-(x) dx < +\infty$.
\end{proof}


\begin{corollary}\label{lifting corollary 2}
	If $f \in C^k(\R)$ and $\int_{\R} |f(x)| dx < +\infty$, then $\ext{h}{f}$ is \civint\ over $\civita_{fin}$ for every $0 \leq h \leq k$ and
	$$
	\int_{\civita_{fin}}^M \ext{h}{f} = \int_{\R} f(x) \ dx.
	$$
\end{corollary}
\begin{proof}
	It is a consequence of the definition of $\int_{\civita_{fin}}^M \ext{h}{f}$ and of the property that $\int_{[a,b]_\civita} \ext{h}{f} = \int_{[a,b]} f(x) \ dx$ for every $a, b \in\R$.
\end{proof}

\subsection{A restricted integration by parts for the \civintegral}
Notice that the \civintegral \ does not satisfy the fundamental theorem of calculus.
A counterexample is obtained as follows: define a function $F : \civita \rightarrow \civita$ by posing
\begin{equation}\label{eqn extl}
	F(x) = \left\{
	\begin{array}{ll}
	0 & \text{if } x<0 \text{ or } x\in\mu(0)\\
	1 & \text{if } x>0 \text{ and } x\not\in\mu(0).
	\end{array}
	\right.
\end{equation}
Then $F$ is continuous, $F'(x)=0$ for all $x\in\civita$, and $F$ is not $m$-measurable over any bounded interval that includes $\mu(0)$.
However, $F$ is \civint\ on every interval $[a,b]_\civita \subset \civita_{fin}$.
Consequently, the fundamental theorem of calculus fails, since e.g.\ if $a<0$ and $b>0$ one has
$$
	F(b)-F(a) = 1 \ne 0 = \int_{[a,b]_\civita}^M F'.
$$
Consequently, an integration by parts formula fails for the real-valued integral.

These drawbacks occur once a sufficiently large class of functions are integrable.
In fact, it is well-known that in non-Archimedean field extensions of $\R$ derivable functions with null derivative need not be constant, and that functions with positive derivative need not be increasing.
For more details on this issue where $\F=\civita$, we refer for instance to \cite{analysislcf,berz+shamseddine2003} and references therein.
Notice also that in fields of hyperreal numbers this problem is overcome by only working with internal functions, and that the hyperreal counterpart of function \eqref{eqn extl} is external.

Despite these limitations, we can still establish a reduced version of the fundamental theorem of calculus and of an integration by parts formula.
As expected, these results can only be obtained for measurable functions in the sense of Definition \ref{measurable functions}.

\begin{proposition}\label{thm fondamentale del calcolo}
	If $a, b \in \civita_{fin}$, $a<b$ and if $f, g\in \L^1([a,b]_\civita)$, then for every $F \in \L^1([a,b]_\civita)$ such that $F'(x) = f(x)$ for all $x\in[a,b]_\civita$,
	$$\int_{[a,b]_{\civita}}^M f = \sh (F(b) - F(a)).$$
\end{proposition}
\begin{proof}
	As a consequence of Proposition 3.24 of \cite{lpcivita}, for $f$ and $F$ satisfying the hypotheses we have
	$$\int_{[a,b]_{\civita}} f = F(b) - F(a).$$
	The desired equality can then be obtained as a consequence of Proposition \ref{prop L^1 implica Loeb 1}.
\end{proof}

As a consequence of the above result, we also get a restricted integration by parts for the $L$-integral.

\begin{proposition}\label{integration by parts}
	If $a, b \in \civita_{fin}$, $a<b$ and if $f, g\in \L^1([a,b]_\civita)$, then
	$$\sh(f(b)g(b)-f(a)g(a)) = \int_{A}^M f'g + \int_A^M fg'.$$
\end{proposition}
\begin{proof}
	Recall that, if $f$ and $g$ are analytic over the Levi-Civita field, then $(fg)'=f'g+g'f$ (see e.g.\ Theorem 5.2 of \cite{calculusnumerics}).
	Thanks to Proposition \ref{thm fondamentale del calcolo}, we deduce that
	$$
	\sh(f(b)g(b)-f(a)g(a)) = \int_{[a,b]_\civita}^M f'g + \int_{[a,b]_{\civita}}^M g'f,
	$$
	as desired.
\end{proof}

\subsection{Delta-like \civint\ functions and their derivatives}\label{sec applications}

Flynn and Shamseddine recently showed that it is possible to represent the Dirac distribution with some measurable functions defined over the Levi-Civita field \cite{shamflin2}.
Their results are mostly concerned with the duality between the so-called delta functions and analytic functions over $\civita$.
A duality with real continuous functions is developed in \cite{lpcivita}, where the discussion is extended also to the derivatives of the Dirac distribution.
However, the proposed approach is somewhat cumbersome, since it is not possible to represent real continuous functions that are not analytic with measurable functions over the Levi-Civita field (see Proposition 3.16 and Lemma 3.17 of \cite{lpcivita}).

This drawback is overcome by using the \civintegral\ defined in Section \ref{sec integral}. As an example, we now discuss the representation of the Dirac distribution and of its derivatives with \civint\ functions on $\civita$.
It is relevant to compare the treatment with the real-valued integral with the one obtained with weak limits of $\m$-measurable functions, discussed in detail in \cite{lpcivita}.

We extend the definition of Dirac-like measurable functions given in \cite{lpcivita} to that of Dirac-like \civint\ functions.

\begin{definition}\label{def dirac-like}
	A \civint\ function $\delta_r:\civita_{fin}\rightarrow\civita$ is Dirac-like at $r \in A$ iff
	\begin{enumerate}
		\item $\delta_{r}(x) \geq 0$ for all $x \in A$;
		\item there exists $h\in M_o$, $h>0$ such that $\supp \delta_{r} \subseteq [r-h,r+h]_{\civita} \subseteq A$;
		\item $\int_{\civita_{fin}}^M \delta_r  = 1$.
	\end{enumerate}
\end{definition}

Notice that, as a consequence of Proposition \ref{prop L^1 implica Loeb 1}, Dirac-like measurable functions are also Dirac-like \civint\ functions.

With this definition and thanks to Theorem \ref{lifting thm}, it is possible to show that Dirac-like \civint\ functions are good representatives of the real Dirac distribution.

\begin{proposition}\label{prop dirac-like}
	Let $r \in \R$.
	For all Dirac-like \civint\ functions $\delta_r$ and for all $f \in C^0(\R)$,
	$$
	\int_{\civita_{fin}}^M \left(\delta_r \cdot \ext{0}{f}\right) = f(r).
	$$
\end{proposition}
\begin{proof}
	The integral is well-defined thanks to Corollary \ref{lifting corollary 2}.
	Notice that $\ext{0}{f}$ is constant over $\mu(r)$.
	Then we have
	$$
	\int_{\civita_{fin}}^M \left(\delta_r \cdot \ext{0}{f}\right) = f(r) \cdot \int_{\civita_{fin}}^M \delta_r = f(r),
	$$
	as desired.
\end{proof}

In a similar way it is possible to represent the derivatives of the Dirac distribution.

\begin{proposition}\label{prop derivata dirac-like}
	Let $r \in \R$.
	For all Dirac-like \civint\ functions $\delta_r$, if $\delta_r \in \L^1(\civita) \cap C^k(\civita)$, then for all $f \in C^n(\R)$ with $n\geq k$, and for all $k \leq j \leq n$,
	$$
	\int_{\civita_{fin}}^M \left(\delta_r^{(k)} \cdot \ext{j}{f}\right) = (-1)^k f^{(k)}(r).
	$$
\end{proposition}
\begin{proof}
	The integral is well-defined thanks to Corollary \ref{lifting corollary 2}.
	
	Moreover, if $h\in M_o$, $h>0$ satisfies $\supp \delta_{r} \subset [r-h,r+h]_{\civita}$, then
	$\ext{j}{f} \in \L^1([r-h,r+h]_{\civita})$.
	In addition, $\delta_r^{(m)}(r\pm h)=0$ for every $0\leq m \leq k$.
	Taking into account these equalities, the restricted integration by parts formula established in Proposition \ref{integration by parts} ensures that
	$$
	\int_{\civita_{fin}}^M \left(\delta_r^{(k)} \cdot \ext{j}{f}\right)
	=
	\int_{[r-h,r+h]_{\civita}}^M \left(\delta_r^{(k)} \cdot \ext{j}{f}\right)
	=
	(-1)^k \int_{[r-h,r+h]_{\civita}}^M \left(\delta_r \cdot \ext{j}{f}^{(k)}\right)
	$$
	for every $0\leq m \leq k$.
	By Proposition \ref{prop dirac-like},
	$$
	(-1)^k \int_{[r-h,r+h]_{\civita}}^M \left(\delta_r \cdot \ext{j}{f}^{(k)}\right)
	\simeq
	(-1)^k \cdot \sh{\left(\ext{j}{f}^{(k)}(r)\right)}.
	$$
	Our hypotheses over $j$ and Theorem 89 of \cite{berz} entail that $\sh{\left(\ext{j}{f}^{(k)}(r)\right)} = f^{(k)}(r)$, so the proof is concluded.
\end{proof}

Notice that the above proposition is false if $j < k$. Under this hypothesis, $\ext{j}{f}$ is locally a polynomial of degree at most $j$, so that $\ext{j}{f}^{(k)}(x) = 0$ for all $x \in\civita_{fin}$.

A comparison of propositions \ref{prop dirac-like} and \ref{prop derivata dirac-like} with their counterparts in \cite{lpcivita} shows the advantage of the use of the real-valued integral for the representation of real distributions in the Levi-Civita field.
Significantly, in \cite{lpcivita} we were not able to define a $C^\infty$ structure over the spaces $\L^p \cap C^\infty(\civita)$ with the duality given by the non-Archimedean integral.
However, the results discussed in this section suggest that it is possible to do so by using the duality induced by the real-valued integral. 

\subsection*{Acknowledgements}
Alessandro Berarducci and Mauro Di Nasso provided insightful comments to some ideas presented in Section \ref{sec loeb}.


\begin{thebibliography}{1}
	
	\bibitem{ultrafunctions1} V.\ Benci \emph{Ultrafunctions and generalized solutions}, in: Advanced Nonlinear
	Studies, 13 (2013), 461–486, arXiv:1206.2257.
	
	\bibitem{bbd} Benci, V.; Bottazzi, E.; Di Nasso, M.
	\emph{Elementary numerosity and measures}, Journal of Logic
		and Analysis \textbf{6} (2014), Paper 3, 14 pp.
	
	\bibitem{bbd2} Benci, V.; Bottazzi, E.; Di Nasso, M.  \emph{Some
	applications of numerosities in measure theory},
	Rendiconti Lincei-Matematica E Applicazioni \textbf{26}
	(2015), no.\;1, 37--48.
	
	\bibitem{omega} V. Benci; L. Horsten; S. Wenmackers, \emph{Infinitesimal probabilities},  British Journal for the Philosophy of Science 69 (2018), 509--552.
	
	\bibitem{berarducci-otero} A. Berarducci, M. Otero, \emph{An additive measure in o-minimal expansions of fields}, in Quarterly Journal of Mathematics, vol. 55, no. 4, pp. 411-419, Dec. 2004, doi: 10.1093/qmath/hah010.
	
	\bibitem[Berz(1992)]{berz} M.\ Berz, \emph{Analysis on a Nonarchimedean Extension of the Real Numbers}. Lecture Notes, 1992.
	
	\bibitem{calculusnumerics} M.\ Berz, \emph{Calculus and numerics on Levi-Civita fields}. In M. Berz, C. Bischof, G. Corliss, and A. Griewank, editors, \emph{Computational Differentiation: Techniques, Applications, and Tools}, pages 19–35, Philadelphia, 1996. SIAM.
	
	\bibitem[Berz\&Shamseddine(2010)]{analysislcf} M. Berz, K. Shamseddine,
	\emph{Analysis on the Levi-Civita field, a brief overview},
	Contemporary Mathematics, 508, 2010, pp.\ 215-237.
	
	\bibitem[Berz\&Shamseddine(2005)]{reexp} M. Berz, K. Shamseddine, \emph{Analytical properties of power series on Levi-Civita fields}, Annales Mathématiques Blaise Pascal, Volume 12 n 2, 2005, pp.\ 309-329. 
	
	
	\bibitem{bottazzi} E. Bottazzi, \emph{A transfer principle for the continuation of real functions to the Levi-Civita field}, p-Adic Numbers, Ultrametric Analysis, and Applications, issue 3, vol 10 (2018).
	
	\bibitem{grid functions} E. Bottazzi, \emph{Grid functions of nonstandard analysis in the theory of distributions and in partial differential equations}, Advances in Mathematics 345 (2019): 429-482.
	
	\bibitem{illposed}  E.\ Bottazzi, \emph{A grid function formulation of a class of ill-posed parabolic equations}, Journal of Differential Equations,
	Volume 271 (2021), Pages 39-75.
	
	\bibitem{lpcivita} E. Bottazzi, \emph{Spaces of measurable functions on the Levi-Civita field}, Indagationes Mathematicae,
	Volume 31, Issue 4 (2020), Pages 650-694.
	
	\bibitem{tesi}
	E. Bottazzi,
	\emph{$\Omega$-Theory: Mathematics with Infinite and Infinitesimal Numbers},  (2012),
	Master thesis, University of Pavia (Italy).
	
	\bibitem{prusso}
	E. Bottazzi, M. Katz, \emph{Internality, transfer and infinitesimal modeling of infinite processes}, to appear in Philosophia Mathematica.
	
	
	
	\bibitem{colombeau 1983} J.\ F.\ Colombeau,
	\emph{A general multiplication of distributions}, Comptes Rendus Acad.\ Sci.\ Paris 296
	(1983) 357–360, and subsequent notes presented by L.\ Schwartz.
	
	\bibitem{integral} Costin, O.; Ehrlich, P.; Friedman, H.
	\emph{Integration on the surreals: a conjecture of Conway, Kruskal
	and Norton}, preprint (2015).  See
	\url{https://arxiv.org/abs/1505.02478}.
	
	\bibitem{cutland controls} N.\ J.\ Cutland, \emph{Infinitesimal Methods in Control Theory: Deterministic and Stochastic}, Acta Applicandae Mathematicae 5 (1986), pp.\ 105--135.
	
	\bibitem{loeb} N.\ J. Cutland, \emph{Loeb measure theory}, in \emph{Loeb Measures in Practice: Recent Advances}, Springer Berlin Heidelberg (2000): 1--28.
	
	\bibitem{cutloeb} N.\ J. Cutland, \emph{Nonstandard Measure Theory and its Applications}, Bulletin of the London Mathematical Society, 15: 529-589. doi:10.1112/blms/15.6.529
	
	\bibitem{eskew} M.\ Eskew, \emph{Integration via ultrafilters} (2020), preprint available at https://arxiv.org/abs/2004.09103
	
	\bibitem{shamflin2} D.\ Flynnn, K.\ Shamseddine, \emph{On Integrable Delta Functions on the Levi-Civita Field}, p-Adic Numbers, Ultrametric Analysis and Applications, 2018, 10.1: 32-56.
	
	
	
	
	
	
	\bibitem{fornasiero} A. Fornasiero, \emph{Integration on Surreal Numbers}, PhD thesis, 2004.
	
	\bibitem{fornasiero hausdorff}  A. Fornasiero, E. Vasquez Rifo, \emph{Hausdorff measure on o-minimal structures}, The Journal of Symbolic Logic, Vol. 77, No. 2 (JUNE 2012), pp. 631-648.
	
	\bibitem[Goldblatt(1998)]{go} R.\ Goldblatt, \emph{Lectures on the Hyperreals: An Introduction to Nonstandard Analysis}, vol. 188 of Graduate Texts in Mathematics, Springer, New York, 1998. Zbl 0911.03032. MR 1643950. DOI 10.1007/978-1-4612-0615-6. 213
	
	\bibitem{henson}
	C.W. Henson (1972),
	\emph{On the nonstandard representation of measures},
	Trans. Amer. Math. Soc., vol. 172, pp. 437--446.
	
	\bibitem{kaiser} T. Kaiser, \emph{Lebesgue measure and integration theory on non-archimedean real closed fields with archimedean value group},
	Proceedings of the London Mathematical Society 116 (2018), no. 2, 209-247.
	
	
	
		
	\bibitem[Levi-Civita(1892)]{civita1} T.\ Levi-Civita,
	\emph{Sugli infiniti ed infinitesimi attuali quali elementi analitici},
	Atti Ist. Veneto di Sc., Lett. ed Art., 7a (4) (1892), p. 1765.
	
	\bibitem[Levi-Civita(1898)]{civita2}T.\ Levi-Civita,
	\emph{Sui numeri transfiniti},
	Rend. Acc. Lincei, 5a (7) (1898), pp. 91-113.
	
	\bibitem{loeb-orig} P.\ A. Loeb, \emph{Conversion from nonstandard to standard measure spaces and applications in probability theory}, Transactions of the American Mathematical Society, 211(1975), 113-22. doi:\url{10.2307/1997222}.
	
	
	\bibitem{moreno} H.\ M.\ Moreno, \emph{Non-measurable sets in the Levi-Civita field}, in \emph{Advances in Ultrametric Analysis: 12th International Conference on P-adic Functional Analysis, July 2-6, 2012, University of Manitoba, Winnipeg, Manitoba, Canada.} American Mathematical Soc., 2013.
	
	\bibitem{payne} S.\ Payne, \emph{Topology of nonarchimedean analytic spaces and relations to complex algebraic geometry}, Bull. Amer. Math. Soc. 52 (2015), 223-247.
	
	\bibitem{pugh} C.\ C.\ Pugh, \emph{Real Mathematical Analysis}, Springer International Publishing, 2015.
	
	\bibitem{robi} A.\ Robinson, \emph{Non-standard analysis}, Nederl. Akad. Wetensch. Proc. Ser. A 64 = Indag. Math. 23 (1961), 432–440
	
	\bibitem{robi1} A.\ Robinson, \emph{Non-standard analysis}, North-Holland Publishing, Amsterdam, 1966.
	
	\bibitem{Shamseddinephd} K.\ Shamseddine, \emph{New Elements of Analysis on the Levi-Civita Field}, PhD thesis, Michigan State University, East Lansing, Michigan, USA, 1999. Also Michigan State University
	report MSUCL-1147
	
	\bibitem{shamseddine2012} K.\ Shamseddine,  \emph{New results on integration on the Levi-Civita field}, Indagationes Mathematicae, 24(1) 2013, pp.199-211.
	
		
	\bibitem{convergence} K.\ Shamseddine, M.\ Berz, \emph{Convergence on the Levi-Civita field and study of power series}, Proc. Sixth International Conference on Nonarchimedean Analysis, pages 283–299, New
	York, NY, 2000. Marcel Dekker.
	
		
	\bibitem{berz+shamseddine2003} K.\ Shamseddine, M.\ Berz, \emph{Measure theory and integration on the Levi-Civita field}, Contemporary Mathematics, 319, 2003, pp.369-388.
	
	\bibitem{shamflin1} K.\ Shamseddine, D.\ Flynn, \emph{Measure theory and Lebesgue-like integration in two and three dimensions over the Levi-Civita fiel}, Contemporary Mathematics, 665, 2016, pp. 289-325.
	
	\bibitem{todorov} T.\ D.\ Todorov, H.\ Vernaeve, \emph{Full algebra of generalized functions and non-standard asymptotic analysis}, J Log Anal 1: 205 (2008). doi:10.1007/s11813-008-0008-y
	
	\bibitem{watt} Wattenberg, F.  \emph{Nonstandard measure theory.
	Hausdorff measure}, Proceedings of the American
		Mathematical Society \textbf{65} (1977), no.\;2, 326--331.
	
	\bibitem{Yeh}
	J. Yeh,
	\emph{Real Analysis, Theory of Measure and Integration}  (2006),
	World Scientific Publishing Co. Pte. Ltd.
\end{thebibliography}
\end{document}